\documentclass[aap,MSNbibl,citesort,dvips]{arximspdf}

\doi{10.1214/11-AAP833} 
\volume{22}
\issue{6}
\pubyear{2012}
\firstpage{2210}
\lastpage{2239}

\makeatletter
\newcommand{\eqref}[1]{(\ref{#1})}

\newcommand{\eee}{{\mathrm{e}}} 

\newcommand{\Probability}{\operatorname{Pr}}

\newcommand{\Expectation}{\mathrm{E}}
\newcommand{\CA}[2]{{\mathcal{A}^{#1}_{#2}}}
\newcommand{\CE}[2]{{\mathcal{E}^{#1}_{#2}}}

\def\E{\mathcal{E}}
\def\C{\mathcal{C}}
\def\R{\mathcal{R}}
\def\eps{\varepsilon}
\def\ls{c_{\mathrm{sob}}}
\def\gap{c_{\mathrm{gap}}}
\def\D{\mathcal{D}}
\def\var{\operatorname{Var}}
\def\ent{\operatorname{Ent}}
\def\twarm{{T_{\mathrm{w}}}}
\def\twarmp{{T'_{\mathrm{w}}}}
\def\cb{{c_2}}
\def\epsl{\varepsilon}
\def\epsu{\delta}

\newcommand{\Tmix}{{T_{\mathrm{mix}}}}
\newcommand{\Trel}{{T_{\mathrm{relax}}}}
\newcommand{\Tlev}{{T_2}}

\newtheorem{theorem}{Theorem}
\newtheorem{lemma}[theorem]{Lemma}
\newtheorem{proposition}[theorem]{Proposition}
\newtheorem{observation}[theorem]{Observation}

\newproclaim{definition}[theorem]{Definition}
\newproclaim{remark}{Remark}
\makeatother

\begin{document}
\begin{frontmatter}

\title{Phase transition for the mixing time of the Glauber dynamics for
coloring regular~trees\thanksref{T1}}
\runtitle{$\!\!\!$The mixing of Glauber dynamics on coloring regular trees}

\thankstext{T1}{Supported in part by NSF Grants DMS-07-01043, CCF-0830298 and CCF-0910584.}

\begin{aug}
\author[a]{\fnms{Prasad} \snm{Tetali}\ead[label=e1]{tetali@math.gatech.edu}},
\author[b]{\fnms{Juan C.} \snm{Vera}\ead[label=e2]{j.c.veralizcano@uvt.nl}},
\author[c]{\fnms{Eric} \snm{Vigoda}\ead[label=e3]{vigoda@gatech.edu}}
\and
\author[c]{\fnms{Linji} \snm{Yang}\corref{}\ead[label=e4]{ljyang@gatech.edu}}
\runauthor{Tetali, Vera, Vigoda and Yang}
\affiliation{Georgia Institute of Technology, Tilburg University,
Georgia Institute of Technology and Georgia Institute of Technology}

\address[a]{P. Tetali\\
School of Mathematics and\\
\quad School of Computer Science\\
Georgia Institute of Technology\\
Atlanta, Georgia 30332\\
USA\\
\printead{e1}}

\address[b]{J. C. Vera\\
Department of Econometrics and\\
\quad Operations Research\\
Tilburg University\\
5000 LE Tilburg\\
The Netherlands\\
\printead{e2}}

\address[c]{E. Vigoda\\
L. Yang\\
School of Computer Science\\
Georgia Institute of Technology\\
Atlanta, Georgia 30332\\
USA\\
\printead{e3}\\
\phantom{E-mail: }\printead*{e4}}

\end{aug}

\received{\smonth{8} \syear{2010}}
\revised{\smonth{9} \syear{2011}}


\begin{abstract}
We prove that the mixing time of the Glauber dynamics for random
\mbox{$k$-colorings} of the complete tree with branching factor $b$ undergoes
a phase transition at $k=b(1+o_b(1))/\ln{b}$.
Our main result shows nearly sharp bounds on the
mixing time of the dynamics on
the complete tree with $n$ vertices
for $k=Cb/\ln{b}$ colors with constant $C$.
For $C\geq1$ we prove the mixing time is $O(n^{1+o_b(1)}\ln{n})$.
On the other side, for $C< 1$ the mixing time
experiences a slowing down; in particular, we prove
it is $O(n^{1/C + o_b(1)}\ln{n})$ and $\Omega(n^{1/C-o_b(1)})$.
The critical point $C=1$ is interesting since it coincides (at least up
to first
order) with the so-called reconstruction threshold which was
recently established by Sly.
The reconstruction threshold has been of considerable interest recently
since it appears to have close connections to the efficiency of certain local
algorithms, and this work was inspired by our attempt to understand
these connections
in this particular setting.
\end{abstract}

%


\end{frontmatter}

\section{Introduction}
\label{introduction}

There has been considerable interest in recent years in understanding the
mixing time of Markov chains arising from single-site updates (known
as Glauber dynamics) for sampling spin systems on finite graphs.
The Glauber dynamics is
well studied both for its computational purposes, most immediately its
use in Markov chain Monte Carlo (MCMC) algorithms,
and for its physical motivation as a model of how physical systems
reach equilibrium. Several works in this topic focus on exploring the
dynamical and spatial connections between the mixing time and
equilibrium properties
of the spin system. A notable example of such equilibrium properties is the
uniqueness of the infinite volume Gibbs measure, which very roughly
speaking corresponds to
the influence of a worst-case boundary condition. Recently a related
weaker notion known
as the \textit{reconstruction threshold} has been the focus of
considerable study.
Reconstruction considers the influence of a ``typical'' boundary
condition (we define
it more precisely momentarily).

Much of the recent interest in reconstruction stems
from its conjectured connections to the efficiency of local algorithms
on trees and
tree-like graphs, such as sparse random graphs.
The Glauber dynamics is one particular
example of such a local algorithm; another important example is the
class of
belief propagation algorithms.
The work of Achlioptas and Coja-Oghlan~\cite{ACO}
gives strong evidence
for the ``algorithmic barriers'' that arise in the reconstruction phase
for several constraint satisfaction problems,
including colorings, on sparse random graphs.
In this paper we show that the mixing time of the Glauber dynamics for random
colorings of the complete tree undergoes a
phase transition, and the critical point appears to coincide with the
reconstruction threshold, at least up to a first order term.

We study the heat-bath version of the Glauber dynamics on the complete
tree with branching factor $b$
for the case of (proper vertex) $k$-colorings.
Proper colorings correspond in the physics community
to the zero-temperature limit of the anti-ferromagnetic Potts model,
and the infinite complete tree is known as the Bethe lattice. Let $\C
=\{
1,2,\ldots,k\}$ denote the set of $k$ colors, and $T_\ell=(V,E)$ denote
the complete tree with branching factor $b$,
height $\ell$ and $n$ vertices.
We are looking at the set $\Omega$ of proper vertex $k$-colorings
which are
assignments $\sigma\dvtx V\rightarrow\C$ such that for all $(v,w)\in
E$ we
have $\sigma(v)\neq\sigma(w)$.
The Glauber dynamics for colorings is a Markov chain $(X_t)$ whose
state space
is $\Omega$ and transitions $X_t\rightarrow X_{t+1}$ are defined as follows:
\begin{itemize}
\item Choose a vertex $v$ uniformly at random.
\item For all $w\neq v$, set $X_{t+1}(w) = X_t(w)$.
\item Choose $X_{t+1}(v)$ uniformly at random from its set of
available colors $\C\setminus X_t(N(v))$ where
$N(v)$ denotes the neighbors of $v$.
\end{itemize}
For the complete tree, when $k\geq3$
the dynamics is ergodic where the unique stationary
distribution is the uniform distribution over $\Omega$.
The mixing time is the number of steps, from the worst initial state,
to reach
within variation distance $\leq1/2\eee$ of the stationary distribution.
We also consider the relaxation time which is the inverse of the
spectral gap
of the transition matrix. We formally define these notions in Section
\ref{preliminaries}.

For general graphs of maximum degree $b$, the Glauber dynamics is ergodic
when $k\geq b+2$ and the best result for arbitrary graphs proves
$O(n^2)$ mixing time when
$k>11b/6 $~\cite{Vigoda}. There are a variety of improvements for
classes of graphs with
high degree or girth (see~\cite{FriezeVigoda} for a survey) and
recently, Mossel and Sly~\cite{MosselSly}
proved polynomial mixing time for sparse random graphs $G(n,d/n)$, for
constant $d>1$,
for some constant number of colors.

There are two phase transitions of primary interest in the tree $T_\ell
$---uniqueness and reconstruction.
These phase transitions are realized by analyzing the influence of the
boundary condition, which in the case of tree corresponds
to fixing the coloring of the leaves. We say uniqueness holds if for
all boundary conditions,
if we consider the uniform distribution conditional on the boundary condition,
the influence at the root decays in the limit $\ell\rightarrow\infty$
(i.e., the root is uniformly distributed over the set $\C$ in the limit).
Jonasson~\cite{Jonasson} established that the uniqueness threshold is
at $k=b+2$.
When $k\leq b+1$ it is not hard
to see that there are boundary conditions which, in fact, ``freeze''
the root;
moreover, the Glauber dynamics is not even ergodic in the case when $k
= b+2$.
Martinelli et al.~\cite{MSW:soda} analyzed the Glauber dynamics on the
tree $T_\ell$ with
a fixed boundary condition. They proved a bound of $O(n\log{n})$ on the
mixing time when
$k\geq b+3$ for any boundary condition.

The reconstruction threshold corresponds to the influence of a random
boundary condition.
In particular, we first choose a random coloring of $T_\ell$, the
colors of the leaves are
fixed, and we rechoose a random coloring for the internal tree from
this conditional distribution.
Reconstruction is said to hold if the leaves have a nonvanishing (as
$\ell\rightarrow\infty$) influence on
the root in expectation.
We refer to
the reconstruction threshold as the critical point for the transition
between the reconstruction and nonreconstruction phases.
It was recently established by Sly that the reconstruction threshold
occurs at $k=b(1+o(1))/\ln{b}$
\cite{Sly,BVVW}.

A general connection between reconstruction and the convergence time of
the Glauber dynamics
was shown by Berger et al.~\cite{BKMP} who showed, for general spin systems,
that $O(n)$ relaxation time on the complete
tree (without boundary conditions) implies nonreconstruction.
A new work of Ding et al.~\cite{DLP} gives very sharp bounds on the
mixing time of the Glauber dynamics for the Ising model on the complete tree,
and illustrates how it undergoes a phase transition at the reconstruction
threshold. For the case of colorings,
recently Hayes et al.~\cite{HVV} proved polynomial mixing time of the
Glauber dynamics for any
planar graph with maximum degree $b$ when $k>100b/\ln{b}$.
Subsequently, improved results were established for the tree. In particular,
Goldberg et al.~\cite{GJK} proved the mixing time is $n^{\Omega
(b/(k\ln{b}))}$
for the complete tree with branching factor $b$, and Lucier et al.
\cite{Molloy}
proved the mixing time is $n^{O(1+b/(k\ln{b}))}$ for any tree with
maximum degree $b$ and the number of colors $k\ge4$. In a follow-up
paper, Lucier et al.~\cite{Molloy2} further prove the same upper bound
for the case when $k=3$.

Our goal is to understand the relationship between the reconstruction
threshold and
the mixing time. Thus we want to establish a more precise picture than
provided by the
results of~\cite{GJK} and~\cite{Molloy}. Our main result
provides (nearly) sharp bounds on the mixing time and relaxation time
of the Glauber dynamics for the complete tree, establishing a phase
transition at
the critical point $k=b(1+o_b(1))/\ln{b}$. Our proofs build upon the
approaches used by~\cite{GJK} and~\cite{Molloy}.\vadjust{\goodbreak}

\begin{theorem}
\label{main-theorem}
For all $C>0$, there exists $b_0$ such that,
for all $b>b_0$,
for $k=Cb/\ln{b}$, the Glauber dynamics on the
complete tree $T$ on $n$ vertices with branching factor $b$ and height
$H = \lfloor\log_b n \rfloor$ satisfies
the following:
\begin{longlist}
\item[(1)] \textit{For $C\geq1$,}
\label{thm:above}
\begin{eqnarray*}
\Omega\bigl(n\ln{n}/(b \operatorname{poly}(\log{b}))\bigr) & \le
&\Tmix
\le O\bigl(n^{1+o_b(1)}\ln{n}\bigr),
\\
\Omega(n) & \le& \Trel\le O\bigl(n^{1+o_b(1)}\bigr);
\end{eqnarray*}
\item[(2)] \textit{For $C< 1$,}
\label{thm:below}
\begin{eqnarray*}
\Omega\bigl(n^{1/C - o_b(1)}\bigr) & \le& \Tmix\le O\bigl(n^{1/C
+ o_b(1)}\ln{n}\bigr),
\\
\Omega\bigl(n^{1/C - o_b(1)}\bigr) & \le& \Trel\le O\bigl(n^{1/C
+ o_b(1)}\bigr),
\end{eqnarray*}
where the $o_b(1)$ functions are $O(\ln{\ln{b}}/\ln{b})$ for the upper
bounds, $b^{1-1/C}/C$ for the lower bounds when $1/2 < C <1$ and
exactly zero for the lower bounds when $0 < C \le1/2$.
The constants in the $\Omega(\cdot)$ and $O(\cdot)$ are universal constants.
\end{longlist}
\end{theorem}

\begin{remark*}
When $C \ge1$, the lower bound of the mixing time is proved by Hayes
and Sinclair~\cite{HS} in a more general setting, and for the
particular case of the heat-bath version of the Glauber
dynamics on the complete tree, we believe it can be
improved to $\Omega(n\ln{n}/\operatorname{poly}(\log{b}))$ by the
same proof.
The lower bound of the relaxation time simply follows from the fact
that the probability of selecting a specific vertex
to recolor in one step of the dynamics is $1/n$.
Note, the results of Berger et al.~\cite{BKMP} imply a lower bound of
$\Trel\geq\omega(n)$
for the case $C<1$ since reconstruction holds in this region.
\end{remark*}

Our result extends to more general $k$ and $b$, thereby refining the general
picture provided by~\cite{GJK} and~\cite{Molloy}.
\begin{theorem}
\label{main-theorem2}
There exists $b_0$ such that, for all $k, b$ satisfying $b/(k\ln b) >
2$ and $b>b_0$,
the Glauber dynamics on the complete tree of $n$ vertices with
branching factor $b$
satisfies the following:
\begin{eqnarray*}
\Omega\bigl(n^{b/(k\ln b)}\bigr) & \le& \Tmix\le O\bigl
(n^{b/(k\ln b) + \gamma
}\ln{n}\bigr),
\\
\Omega\bigl(n^{b/(k\ln b)}\bigr) & \le& \Trel\le O\bigl
(n^{b/(k\ln b) + \gamma}\bigr),
\end{eqnarray*}
where
\[
\gamma= \gamma(b) = 1 - \frac{\ln{k}}{\ln{b}} + \frac{\ln\ln
{b}}{\ln
{b}} + \frac{O(1)}{\ln b}
\]
is at most a small constant.
\end{theorem}

\begin{remark*}
The constants in the $\Omega(\cdot)$ and $O(\cdot)$ of Theorem \ref
{main-theorem2} are universal
constants.
Also, note that when $k = b^\alpha$ for constant $\alpha< 1$, then
$\lim_{b\rightarrow\infty} \gamma= 1 - \alpha$,
and when $k$ is constant, then $\lim_{b\rightarrow\infty} \gamma= 1$.
\end{remark*}

\section{Proof overview}
\label{proofideas}
We now give an outline of the proofs of Theorem~\ref{main-theorem}.
Readers can refer to Section~\ref{preliminaries} for the definitions
and background materials.

\subsection{Upper bounds}

We first sketch the proof approach for upper bounding the mixing time
and relaxation time. Let $G^* = (V,E)$ be the star graph on $b+1$
vertices, that is, the complete tree $T_1$ of height $1$ with $b$
leaves, and $H$ be the height of the complete tree $T_H$, that is, $H =
\lfloor\log_b n \rfloor$. Let $\tau^*$ be the relaxation time of the
Glauber dynamics on the star graph $G^*$
using $k$ colors.

We use the following decomposition result of Lucier and Molloy \cite
{Molloy}, which is an application of the block dynamics technique (see
Proposition 3.4 in~\cite{Martinelli-lecturenotes}) to the Glauber
dynamics on the complete trees combined with an earlier result proved
by Berger et al. which shows that the relaxation time of this special
block dynamics is the same as that of the Glauber dynamics on the star
graph (see Claim 2.9 in~\cite{BKMP}).
\begin{theorem}
\label{thm:decomposition}
The relaxation time $\Trel$ of the Glauber dynamics on the complete
tree of height $H$ with branching factor $b$ satisfies
\[
\Trel\le(\tau^*)^H.
\]
\end{theorem}

Therefore, proving the upper bounds in Theorem~\ref{main-theorem}
reduces to the problem of getting tight upper bounds of the relaxation
time $\tau^*$ of the Glauber dynamics on $G^*$. In~\cite{Molloy}, the
authors used a canonical path argument to bound $\tau^* =
O(b^{2+1/C}k)$ for any $C > 0$.
Instead, here we use two different coupling arguments to show the
following two theorems for $\tau^*$.

\begin{theorem}
\label{thm:star-upper-below}
For any $C < 1$, there exists $b_0 > 0$ such that, for any $b > b_0$,
the mixing and relaxation times of the Glauber dynamics on $G^*$ using
$k = Cb/\ln b$ colors are $O(b^{1/C}\ln^2 b)$. When $C = 1$, the mixing
and relaxation times are $O(b\ln^4 b)$.
\end{theorem}

\begin{theorem}
\label{thm:star-upper-above}
For any $C > 1$, there exists $b_0 > 0$ such that, for any $b > b_0$,
the mixing and relaxation times of the Glauber dynamics on $G^*$ using
$k \ge Cb/\ln b$ colors are $O(b\ln b)$.
\end{theorem}

\begin{remark*}
It can be shown that the relaxation time is actually $O(b)$ when $C >
1$, from our analysis. However,
unless we can also eliminate the constant factors and thereby show a
very sharp bound of at most $b$,
the extra $\ln{b}$ factor makes little difference to the relaxation
time of the dynamics on the whole tree.
\end{remark*}

The most difficult (and also interesting) case turns out to be when
$C\le1$.
We will prove Theorem~\ref{thm:star-upper-below} in Section \ref
{upbelow} and Theorem~\ref{thm:star-upper-above} in Section~\ref{upabove}.
We sketch the high-level idea of the proof of Theorem \ref
{thm:star-upper-below}
in Section~\ref{upper:couplingproofidea}.
Having Theorems~\ref{thm:star-upper-below} and~\ref
{thm:star-upper-above} in hand, we can then apply Theorem \ref
{thm:decomposition} to get\vadjust{\goodbreak} the upper bounds on the relaxation time as
stated in Theorem~\ref{main-theorem}. We get
\begin{displaymath}
\Trel=
\cases{
O(b\ln b)^H = O\bigl(n^{1+(\ln\ln b + O(1))/{\ln b}}\bigr), & \quad$\mbox{if
}C > 1 ,$ \vspace*{2pt}\cr
O(b\ln^4 b)^H = O\bigl(n^{1 + (4\ln\ln b + O(1))/{\ln b}}\bigr), &\quad
$\mbox{if } C=1,$ \vspace*{2pt}\cr
O(b^{1/C}\ln^2 b)^H = O\bigl(n^{1/C + (2\ln\ln b + O(1))/{\ln
b}}
\bigr), &\quad$\mbox{if } C<1 .$}
\end{displaymath}

To then get the desired upper bounds on the mixing time of the whole
tree, we need a slightly more advanced tool, the logarithmic Sobolev
constant of the Markov chain; we define the log-Sobolev constant
formally in the next section along with the other technical preliminaries.
By adapting Theorem~5.7 in Martinelli, Sinclair and Weitz~\cite{MSW} to
our setting of colorings,
we establish and improve (in Section~\ref{logsobolev}) the following
relationship between the inverse of the log-Sobolev constant $\ls^{-1}$
and the relaxation time $\Trel$ of the Glauber dynamics on trees.
\begin{theorem}
\label{thm:LS}
\[
\ls^{-1} \le\Trel\cdot2b\ln(k).
\]
\end{theorem}

Since the inverse of the log-Sobolev constant gives a relatively tight
upper bound on the mixing time [see inequality \eqref{eqn:TmixLS} in
Section~\ref{preliminaries}], using Theorem~\ref{thm:LS} we are able to
complete the proofs of the upper bounds in Theorem~\ref{main-theorem}.

\subsection{Lower bounds}

Our proof of the lower bound in Theorem~\ref{main-theorem} when $C < 1$
builds upon the approach used in~\cite{GJK}. They lower bounded the
relaxation time by upper bounding the conductance of the Glauber
dynamics on the subset $S\subseteq\Omega$ where the root is \textit{frozen}
(meaning that the configuration at the leaves uniquely
determine the color of the root) to some color in $\{1,2,\ldots
,\lfloor
k/2\rfloor\}$. They showed the conductance of $S$ satisfies $\Phi_S =
O(n^{-1/6C})$ when $0 < C < 1/2$, which implies [by \eqref
{eqn:TmixTrel} and \eqref{eqn:lbound-Trel}
in Section~\ref{preliminaries}] that $\Tmix\ge\Omega(\Trel) =
\Omega
(n^{1/6C})$.

We improve their bound on the conductance of $S$ by analyzing
the probability that for a given leaf $z$, in a random coloring $\sigma
$ of the complete tree, the root is frozen and changing the color of
$z$ in $\sigma$ to some other color unfreezes the root.
We prove that the number of such leaves in most colorings that freeze
the root is $O(n^{-1/C+1+o_b(1)})$. Since the probability of recoloring
a specific leaf is $1/n$, then intuitively we have $\Phi_S =
O(n^{-1/C+o_b(1)})$, and hence $\Tmix\ge\Omega(\Trel) = \Omega
(n^{1/C-o_b(1)})$.
A complete analysis of the lower bound is in Section~\ref{lbbelow}, and
in the analysis we will see that the $o_b(1)$ error term is
$b^{1-1/C}/C$ when $1/2 < C <1$ and zero when $C\leq1/2$.

Finally, we will show in Section~\ref{generalization} how all of the proofs
generalize for $k = o(b/\ln b)$, and thus prove Theorem~\ref{main-theorem2}.

\section{Technical preliminaries}
\label{preliminaries}

Let $P(\cdot,\cdot)$ denote the transition matrix of the Glauber
dynamics, and $P^t(\cdot,\cdot)$ denote the $t$-step transition probability.
The total variation distance at time $t$ from initial state $\sigma$ is
defined as
\[
\|P^t(\sigma,\cdot) - \pi\|_{\mathrm{TV}} := \frac{1}{2}\sum_\eta
|P^t(\sigma
,\eta)-\pi(\eta)|.
\]

The mixing time $\Tmix$ for a Markov chain is then defined as
\[
\Tmix= \min_t \Bigl\{ \max_\sigma\{ \|P^t(\sigma,\cdot) - \pi\|
_{\mathrm{TV}}
\}
\le1/2\eee\Bigr\}.
\]

Given two copies, $(X_t)$ and $(Y_t)$, of the Markov chain at time
$t>0$, recall that a (one-step) coupling of $(X_t)$ and $(Y_t)$ is a
joint distribution whose left and right marginals are identical to the
(one-step) evolution of $(X_t)$ and $(Y_t)$, respectively. The Coupling
lemma~\cite{Aldous} (cf. Theorem 5.2 in~\cite{LPW}) guarantees
that if
there is a coupling and time
$t >0$, so that for every pair $(X_0, Y_0)$ of initial states
$\Probability[{X_t \neq Y_t}] \le1/2\eee$ under the coupling, then
$\Tmix\le t$.

Let $\lambda_1 \ge\lambda_2 \ge\cdots\ge\lambda_{|\Omega|}$ be the
eigenvalues of the transition matrix $P$. The spectral gap $\gap$ is
defined as $1-\lambda_2$.
The relaxation time $\Trel$ of the Markov chain is then defined as
$\gap
^{-1}$, the inverse of the spectral gap.
It is an elementary fact that the mixing time gives a good upper bound
on the relaxation time (see, e.g.,~\cite{DGJM} for the following
bound), which
we will use in our analysis.
%
%
\begin{equation}
\label{eqn:TmixTrel}
\Trel= O(\Tmix).
\end{equation}
Note that our definition of relaxation time following~\cite{BKMP,MSW}
is slightly different from the standard definition, the inverse of the
absolute spectral gap (see, e.g., Chapter~13 in~\cite{LPW}), that is,
$(1 - \max\{|\lambda_2|, |\lambda_{|\Omega|}|\})^{-1}$.
It would be easier for us to state the results related to the block
dynamics under our current definition, and it is a standard fact that
by passing to a lazy chain, the two definitions are identical.
Introducing the laziness to the Glauber dynamics only puts an extra
factor of two to the mixing time, and therefore it will not affect our
asymptotic results.

Since we will also work with the logarithmic Sobolev constant of a
(finite) Markov chain, we briefly recall here the variational
definition of both the spectral gap and the log-Sobolev constant.

Let $f$ be a function (vector) from $\Omega$ to $R$, $\pi$ be the
uniform distribution over $\Omega$ and $\mu$ be any probability
distribution over $\Omega$. Let $\D(f)$ be the standard \textit{Dirichlet
form} of the heat-bath Glauber dynamics defined as
\[
\D(f) = \frac{1}{2}\sum_{\sigma} \sum_{\sigma'} \bigl(f(\sigma) -
f(\sigma
')\bigr)^2 \pi(\sigma) P(\sigma, \sigma').
\]

Let $E_{\mu}(f)$ be the average of $f$ under the distribution $\mu$, and
let $\var_{\mu}(f) := E_{\mu}(f^2) - E^2_{\mu}f$ be the corresponding
variance, which can also be written as
\[
\var_{\mu}(f) = \frac{1}{2}\sum_{\sigma} \sum_{\sigma'}
\bigl(f(\sigma) -
f(\sigma')\bigr)^2 \mu(\sigma) \mu(\sigma').
\]

Let $\ent_{\mu}(f) := E_{\mu}(f\log f) - E_{\mu}(f) \log(E_{\mu}(f))$.
When it is clear what the underlying distribution is we will drop the
subscript $\mu$ in the notation $\ent(f)$.

The spectral gap $\gap$ is equivalently defined as (see, e.g., Chapter
13, in~\cite{LPW})
\[
\gap= \inf_f\frac{\D(f)}\var(f),
\]
and the log-Sobolev constant $\ls$ is defined as (see, e.g.,~\cite{DS}),
\[
\ls= \inf_{f\ge0}\frac{\D(\sqrt{f})}{\ent(f)},
\]
where the infimum in both equations is over nonconstant functions $f$.


For the upper bounds on the mixing time of the dynamics on the whole
tree, we also use the following well-known relationship between the
mixing time and the inverse of the log-Sobolev constant (see, e.g.,
\cite{DS} for more details):
%
%
\begin{equation}
\label{eqn:TmixLS}
\Tmix= O\biggl(\ls^{-1} \ln{\ln{\frac{1}{\min_{\sigma\in
\Omega} \{
\pi(\sigma)\}}}}\biggr).
\end{equation}


To lower bound the mixing and relaxation times we analyze
the conductance. The conductance of the Markov chain on $\Omega$ with
transition matrix $P$ is given by $
\Phi=\min_{S\subseteq\Omega} \{\Phi_S\}
$,
where $\Phi_S$ is the conductance of a specific set $S \subseteq
\Omega
$ defined as
\[
\Phi_S=\frac{\sum_{\sigma\in S}\sum_{\eta\in\bar{S}}\pi(\sigma
)P(\sigma
,\eta)}
{\pi(S)\pi(\bar{S})}
.
\]

Thus, a general way to find a good upper bound on the conductance is to
find a set $S$ such that the probability of escaping from $S$ is
relatively small.
The well-known relationship between the relaxation time and the
conductance is established in~\cite{LawlerSokal} and~\cite{SinclairJerrum},
and we will use the form
%
%
\begin{equation}
\label{eqn:lbound-Trel}
\Trel= \Omega(1/\Phi) ,
\end{equation}
for proving the lower bounds.

\section{\texorpdfstring{Upper bound on mixing time for $C\le1$: Proof of Theorem \protect\ref{thm:star-upper-below}}
{Upper bound on mixing time for C <= 1: Proof of Theorem 4}}
\label{upbelow}
In this section, we upper bound
the mixing time of the Glauber dynamics on the star graph $G^* = (V,E)$
when $k=Cb/\ln b$ for any $C \le1$. To be more precise, let $V = \{
r,\ell_1,\ldots,\ell_b\}$, where $r$\vadjust{\goodbreak} refers to the root and $\ell
_1,\ldots,\ell_b$ are the $b$ leaves and $E = \{(r,\ell_1),\ldots
,(r,\ell
_b)\}$.
For convenience, here we let
\[
\eps:= 1/C - 1,
\]
and hence $k = b/((1+\eps)\ln b)$.

We use the maximal one-step coupling, originally studied
for colorings by Jerrum~\cite{Jerrum}, to upper bound the mixing time
of the Glauber dynamics on general graphs.
For a coloring $X\in\Omega$, let $A_{X}(v)$ denote the set of available
colors of $v$ in the coloring $X$, that is,
$A_{\sigma}(v) = \{c\in\C\dvtx \forall u \in N(v), \sigma(u) \neq
c \}$.
The coupling $(X_t,Y_t)$ of the two chains is done by choosing the same
random vertex $v_t$ for recoloring at step $t$ and maximizing the
probability of the two chains choosing the same update for the color of $v_t$.
Thus, for each color $c\in A_{X_t}(v) \cap A_{Y_t}(v)$, with
probability $1/\max\{|A_{X_t}(v)|,|A_{Y_t}(v)|\}$ we set
$X_{t+1}(v)=Y_{t+1}(v)=c$. With the remaining probability, the color
choices for $X_{t+1}(v)$ and $Y_{t+1}(v)$ are coupled arbitrarily.

We prove the theorem by analyzing the coupling in rounds, where each
round consists of $T:=20 b\ln{b}$ steps.
Our main result is the following lemma which says
that in each round, we have a good probability of coalescing (i.e.,
achieving $X_t=Y_t$).

\begin{lemma}
\label{lem:upper-below}
For all $\eps\geq0$, there exists $b_0(\eps)$ such that for all
$b>b_0(\eps)$ if $k = b/((1+\eps)\ln b)$ and $T=20b\ln{b}$ for all
$(x_0,y_0)\in\Omega\times\Omega$, the following holds:
\[
\Probability[{X_{T} = Y_{T}} \mid{X_0 = x_0, Y_0 = y_0} ] \ge
\cases{
{\bigl(20(1+\eps)b^\eps\ln b\bigr)}^{-1}, & \quad$\mbox{if }\eps
> 0 ,$
\vspace*{2pt}\cr
{(20\ln^3 b)}^{-1}, &\quad$\mbox{if } \eps=0.$}
\]
\end{lemma}

It is then straightforward to prove Theorem~\ref{thm:star-upper-below}.
\begin{pf*}{Proof of Theorem \protect\ref{thm:star-upper-below}}
For $\eps>0$, let $p_T := (20(1+\eps)b^\eps\ln b)^{-1}$;
and for $\eps=0$ let $p_T := (20\ln^3 b)^{-1}$.
By repeatedly applying Lemma~\ref{lem:upper-below} we have, for all
$(x_0,y_0)$,
\[
\Probability[{X_{2iT} \neq Y_{2iT}} \mid{X_0 = x_0, Y_0 = y_0} ]
\leq
(1-p_T)^{2i}\leq1/2\eee
\]
for $i=1/p_T$.
Therefore, by applying the coupling lemma, mentioned in Section~\ref
{preliminaries}, the mixing time is
$O((1+\eps)b^{1+\eps}\ln^2 b)$ for $\eps> 0$
and $O(b\ln^4 b)$ for $\eps= 0$.
\end{pf*}

\subsection{Overview of the coupling argument}
\label{upper:couplingproofidea}

Before formally proving Lem\-ma~\ref{lem:upper-below} we give a
high-level overview of its proof.
We will analyze the maximal one-step coupling on the star graph $G^*$.
We say a vertex $v$ ``disagrees'' at time $t$ if $X_t(v) \neq Y_t(v)$,
otherwise we say the vertex $v$ ``agrees.''
We denote the set of disagreeing vertices at time $t$ of our coupled
chains by
\[
D_t = \{v\in V\dvtx X_t(v)\neq Y_t(v)\},
\]
and we use $D_t^L = D_t\setminus\{r\}$ to represent the set of
disagreeing leaves.
When we use the term ``with high probability'' in this section,
it means that the probability goes to 1 as $b$ goes to infinity.


If the coupling selects a leaf $\ell$ to recolor at time $t$, then the
probability that $\ell$ disagrees in $X_t$ and $Y_t$ is at most $1/(k-1)$,
and with probability at least
$(k-2)/(k-1)$, the leaf will use the same color that is chosen
uniformly at
random from $\C\setminus\{X_t(r),Y_t(r)\}$. We also know that if we
simply assign a random color from $\C$ to each leaf, with probability
at least
$\Omega(1/(b^\eps\ln b))$, there is a color in $\C$ that is unused in
any leaf.
This last point hints at the success probability in the statement of Lemma
\ref{lem:upper-below}.

We analyze the $T$-step epoch in three stages. The warm-up round is of length
$\twarm:= 8(b+1)\ln b$ steps.
We will show in Lemmas~\ref{lem:warmup} and~\ref{lem:warmup0} that
with good probability, after the warm up,
all of the leaf disagreements will be of the same form
in the sense that they will have the same pair of colors.

The next stage is of a random length $T_1$, which is defined as the
first time (after $\twarm$) where we are recoloring the root, and
the root has a common available color in $(X_t)$ and $(Y_t)$. We prove
in Lemma~\ref{lem:common-avail} that with probability $\Omega
(1/b^{\eps
}\ln{b})$,
$T_1<4(b+1)\ln{b}$. We then have probability at least $1/2$ of the
root agreeing
after it is updated, and then after at most
$\Tlev:=4(b+1)\ln{b}$ further steps we are likely to coalesce
since we just need to recolor each leaf at least once before the root
changes back
to a disagreement.

\subsection{\texorpdfstring{Coupling argument: Proof of Lemma \protect\ref{lem:upper-below}}
{Coupling argument: Proof of Lemma 7}}

We begin with a basic observation about the maximal one-step coupling.

\begin{observation}
\label{prop:avail}
Let $\C(D_t^L) := \bigcup_{\ell\in D_t^L} \{X_t(\ell), Y_t(\ell)\}
$ denote
the set of colors that appear in the disagreeing leaves at time $t$. Then
$A_{X_t}(r) \oplus A_{Y_t}(r) \subseteq\C(D_t^L)$.
\end{observation}

This is simply because those colors that appear on the leaves with agreements
are both unavailable in $X_t$ and $Y_t$ for the root.
We now analyze the first stage of the $T$-step epoch.

\begin{proposition}
\label{prop:roothit} The probability that in $T_0 = 4(b+1)\ln b$ steps,
the coupling $(X_t,Y_t)$ [or the Glauber dynamics $(X_t)$] will recolor
the root at most $20\ln b$ times and recolor every leaf at least once
is at least $1 - 2b^{-3}$.
\end{proposition}

\begin{pf}
Using the union bound the probability that there is a leaf which is not
recolored in $T_0$ steps is at most
\[
b\biggl(1-\frac{1}{b+1}\biggr)^{4(b+1)\ln b} \le b^{-3}.
\]
Now, let $N$ be the number of times the root is recolored in $T_0$ steps.
The expectation ${\Expectation[{N}]}$ is $4\ln b$. Then, by the\vadjust{\goodbreak}
Chernoff bound
(see, e.g., Theorem~4.5, Part 2 in~\cite{Upfal}),
\[
\Probability[{N \ge20\ln b}] \le\Probability\bigl[{N \ge
(1+4){\Expectation[{N}]}}\bigr] \le b^{-3}.
\]
Therefore the lemma holds by the union bound.
\end{pf}

Then we will prove that in $\twarm= 2T_0$ steps, with high probability
all of the leaf disagreements are of the same type when $\eps> 0$.


\begin{lemma}
\label{lem:warmup}
For any $\eps> 0$ and $k > (1+\eps) b/ \ln b$, for any pair of initial
states $(x_0,y_0)$,
\[
\Probability[{\forall\ell\in D_\twarm^L, X_\twarm(\ell) =
Y_\twarm(r) \wedge Y_\twarm(\ell) = X_\twarm(r)} \mid{x_0, y_0} ]
\ge
1 - O\biggl(\frac{1}{b^{\eps}}\biggr).
\]
\end{lemma}

\begin{pf}
The idea is that if we just look at one chain, say $(X_t)$, then after~$T_0$ steps,
with high probability the root is frozen. Moreover, the
root is likely to continue to be frozen
for the remainder of the $\twarm$ steps since we recolor the root at
most $O(\ln{b})$ times. In the worst case the root is frozen to a
disagreement, say $X_t(r)=2$ and $Y_t(r)=1$.
Then after recoloring a leaf $\ell$ at time $t'$ where $t<t'<\twarm$,
the only possible disagreement
is $X_{t'}(\ell)=1, Y_{t'}(\ell)=2$. Hence, it suffices to recolor each
leaf at
least once.

Let $\E$ be the event that in the first $T_0$ steps, every leaf is
recolored at least once and in another $4(b+1)\ln{b}$ steps, every leaf
is recolored again at least once, and the root is recolored at most
$20\ln b$ times.
We are first going to bound that for $t > T_0$,
%
%
\begin{equation}
\label{eqn:warmup1}
\Probability[{|A_{X_t}(r)| > 1} \mid{\E} ] \le\frac{1}{(1+\eps
)b^\eps\ln b}
:= p_0,
\end{equation}
and the same thing happens for $Y_t$.

Let $G_W$ be the graph with $b$ isolated vertices $\{v_1,\ldots,v_b\}$,
corresponding to the leaves $\{\ell_1,\ldots,\ell_b\}$. Let $(W_t)$
be a
Glauber process on $G_W$ using $k-1$ colors from another color set $\C
_W$. We are going to define $W_0$ and couple $(W_t)$ with $(X_t)$ such
that $|A_{X_t}(r)|= |A_{W_t}| + 1$
at any time $t$, where $A_{W_t} := \{c\in\C_W \dvtx \forall v_i, W_t(v_i)
\neq c\}$.
To do this, for every $t$ we are going to define a bijection $f_t
\dvtx \C
\setminus\{X_t(r)\} \to\C_W$ such that $f_t(X_t(\ell_i)) = W_t(v_i)$
for all $i$. Notice that if such a bijection exists, then
$|A_{X_t}(r)|= |A_{W_t}| + 1$.

At time $t=0$, pick any bijection $f_0$ from $\C_W$ to $\C\setminus\{
X_0(r)\}$. Define $W_0$ by $W_0(v_i) = f(X_0(\ell_i))$ for all $i$. We
will update $f_t$ only when we choose the root to recolor at time $t$
in the coupling of $(W_t)$ and $(X_t)$.
To do the coupling at time $t+1$, we first choose a vertex $v$ in $G^*$
to recolor:
\begin{itemize}
\item If $v = \ell_i$, then we choose a random color $c$ that is
different from $X_t(r)$ to recolor~$v$.
Correspondingly, we choose the vertex $v_i$ in $G_W$ to recolor using
color $f_t(c)$.
\item If $v = r$, then we choose a random color $c$ from $A_{X_t}(r)$
to recolor the root in~$G^*$. Correspondingly, we update the mapping
$f_t$ in the following natural way:
$f_t(X_{t-1}(r))= f_{t-1}(c)$ [and $f_t(c)$ is undefined].
\end{itemize}

Since $(W_t)$ itself is a Glauber process that recolors the vertices of
$G_W$ uniformly at random from $C_W$, conditioning on $\E$, simple
calculations yield that for any $t > T_0$,
\[
\Probability[{|A_{W_t}| \ge1} \mid{\E} ] \le\frac{1}{(1+\eps
)b^\eps\ln b} .
\]
Then \eqref{eqn:warmup1} follows by coupling.

Since the same thing happens for $(Y_t)$, and the root is recolored at
most $20\ln b$ times, then by the union bound, conditioning on $\E$,
the probability that at each time we try to recolor the root after
$T_0$ steps, the root is always frozen in both copies is at least $
1 - (40\ln b)(p_0) = 1 - 40/((1+\eps)b^{\eps}) .
$
Finally, by Proposition~\ref{prop:roothit}, $\E$ happens with high
probability, and hence the lemma holds.
\end{pf}

Note that for the warm-up stage, we need to show, with probability at
least $1/\operatorname{poly}(\log b)$, that for $\eps\ge0$,
all of the leaf disagreements are of the same type in $O(b\ln b)$
steps. This is easier to prove for the $\eps> 0$ case -- that this
happens with high probability, if we run the dynamics for $\twarm=
8(b+1)\ln b$ steps.
For the threshold case when $\eps= 0$, we will prove a slightly weaker
lemma, in the sense that the successful probability will be at least
$\Omega(1/\ln^2 b)$.


\begin{lemma}
\label{lem:warmup0}
Let $\twarmp= T_0 + 2b\ln\ln b$. For $k = b/\ln b$, for any pair of
initial states $(x_0,y_0)$,
\[
\Probability[{\forall\ell\in D_{\twarmp}^L, X_{\twarmp}(\ell) =
Y_{\twarmp }(r) \wedge Y_{\twarmp}(\ell) = X_{\twarmp}(r)} \mid
{x_0,y_0} ] \ge1/(2\ln
^2 b).
\]
\end{lemma}

\begin{pf}
We use a different approach to prove this lemma, since it is not true
that the root will still always be frozen during $\twarmp$ steps with
high probability.

Let $T_0 = 4(b+1)\ln b$. We first prove that after $T_0$ steps, with
high probability, the number of disagreeing leaves is at most $O(\ln
b)$, namely,
%
%
\begin{equation}
\label{eqn:epszero1}
\Probability[{|D_{T_0}^L| \ge4 \ln b} \mid{X_0 = x_0 , Y_0 = y_0}
] \le\frac{2}{b^2}.
\end{equation}

To prove \eqref{eqn:epszero1}, we
construct a simpler process that stochastically upper bounds the number
of disagreements.
We define the following Markov chain $(U_t)$
on 2-colorings of the graph $G_U$ which consists of $b$ isolated
vertices $\{v_1,\ldots,v_b\}$. We view the set of colors as $\{0,1\}$. In
each step, a random vertex $v_i$ is chosen,
then with probability $1/(k-1)$, $v_i$ is
recolored to $1$,
and with probability $1-1/(k-1)$,\vadjust{\goodbreak} $v_i$ is recolored to $0$. Let $D^U_t
= \{v \in\{v_1,v_2,\ldots,v_b\}\dvtx U_t(v)=1\}$.
The initial state $U_0$ is constructed in the following way: for any $i
> 0$, $U_0(v_i) = 1$ if and only if $x_0(\ell_i) \neq y_0 (\ell_i)$.
By associating the $b$ vertices of $G_U$ with the leaves of $G^*$,
we can easily couple
the process $(U_t)$ with $(X_t,Y_t)$ such that $|D^U_t| \ge|D_t^L|$.

Let $\E$ denote the event that all of the vertices of $G_U$ are
recolored at least once in $T_0$ steps.
Note $\Probability[{\E}] \ge1 - 1/b^2$.
Conditioned on $\E$, the expected size of $|D^U_{T_0}|$ is
$b/(k-1)\approx\ln{b}$. Then we have
\begin{eqnarray*}
\Probability[{|D^U_{T_0}| \ge4 \ln b}] &\le& \Probability
[{|D^U_{T_0}| \ge4\ln b} \mid{\E} ] + \Probability[{\E}] \\
& \le& \frac{2}{b^2}.
\end{eqnarray*}
Here, for the last inequality, we have used the Chernoff bounds (see,
e.g., Theorem~4.5 Part 2 in~\cite{Upfal}).
Since $|D^U_t| \ge|D_t^L|$, this proves \eqref{eqn:epszero1}.

Hence, with high probability there are $O(\ln b)$
disagreeing leaves in $G^*$ at time~$T_0$.
Notice that from time $T_0$, if we recolor all of the disagreeing
leaves before we
recolor the root again, then all of the remaining disagreements in the
leaves will be of the same type [more precisely, for such a
leaf $\ell$ that becomes a disagreement at time $t$, we
will have that $X_t(\ell)=Y_{T_0}(r)$ and $Y_t(\ell)=X_{T_0}(r)$], and
this implies the desired conclusion of the lemma.
To this end,
let $\E_2$ be the event that the root is not chosen from
recoloring from time $T_0$ to $\twarmp$.
Let $\E_3$ be the event that each leaf in $D_{T_0}^L$ is
recolored at least once in the interval of times $[T_0,\twarmp]$.
By simple calculations, we have that
%
%
\begin{equation}
\Probability[{\E_2}] \ge\ln^{-2} b,\qquad
\Probability[{\E_3} \mid{\E_2} ] \ge1 - \frac{O(1)}{\ln b}.
\end{equation}
Therefore, conditioned on $|D^L_{T_0}| \le4 \ln b$,
from time $T_0$ to $\twarmp$ with probability at least $2/(3\ln^2 b)$,
both $\E_2$ and $\E_3$ happen, which implies all of the leaf
disagreements will be of the same type at time $\twarmp$.

In conclusion, combining the above bounds with \eqref{eqn:epszero1}, we
proved that with probability at least $1/(2\ln^2 b)$, all of the
uncoupled leaves are of the same type at time~$\twarmp$.
\end{pf}

After we succeed in the warm-up stage, meaning that all of the leaf
disagreements are of the same type, we enter the root-coupling stage,
where we try to couple the root. Let $T_1$ be the first time that there
is a common available color in the root, and the coupling chain selects
the root to recolor, that is,
\[
T_1 := T_1^{XY} =
\min\{t\dvtx A_{X_t}(r) \cap A_{Y_t}(r) \neq\varnothing
\mbox{ and the root $r$ is selected at step $t$}\}.
\]

\begin{lemma}
\label{lem:common-avail}
For $\eps\ge0$, for any pair of initial states $(x_0,y_0)$ where
all of the leaf disagreements are of the same type
[i.e., there is a pair of colors $c_1,c_2$ such that
for all $\ell\in D^L_0$, we have $x_0(\ell) = c_1$ and $y_0(\ell) =
c_2$], we have
\[
\Probability[{T_1^{XY} < 4(b+1)\ln b} \mid{(X_0,Y_0)=(x_0,y_0)} ] >
\frac
{1}{4(1+\eps)b^\eps\ln b}.\vadjust{\goodbreak}
\]
\end{lemma}

\begin{pf}
First of all, by Proposition~\ref{prop:avail}, $|A_{X_0}(r)\oplus
A_{Y_0}(r)| \le2$. We are interested in the time $t$ when there is a
common color available for the root in $(X_t,Y_t)$.

Let $(Z_t)$ be a Glauber process on the graph $G_Z$ of $b+1$ isolated
vertices $\{v_0,v_1,v_2,\ldots,v_b\}$ in which $v_0$ corresponds to the
root and $v_i$ corresponds to the leaves $\ell_i$ for any $i > 0$. The
color set used in the process $(Z_t)$ is $\C_Z = [k]\setminus\{
{c_1},\cb
\}$. In each step, $(Z_t)$ chooses a random vertex and recolors it with
a random color from the set $\C_Z$. Let $T_Z$ be the stopping time on
$Z$, satisfying
\[
T_1^{Z} = \min\{ t > 2(b+1)\ln b\dvtx |A_{Z_t}| \ge1 \mbox{ and
$v_0$ is
selected at the step $t$}\},
\]
where $A_{Z_t} = \{c\in\C_Z\dvtx \forall i\in[1,\ldots ,b], Z_t(v_i)
\neq c\}$
is the set of unused colors in the vertices $\{v_1,v_2,\ldots , v_b\}$.
We want to couple $(Z_t)$ with $(X_t,Y_t)$ in such a way that $T_1^{Z}
\ge T_1^{XY}$ for all the runs, and then if we show that for any
initial state $z_0$, we have
%
%
\begin{equation}
\label{eqn:tz}
\Probability[{T_1^{Z} < 4(b+1)\ln b} \mid{Z_0=z_0} ] > \frac
{1}{4(1+\eps)b^\eps
\ln b}.
\end{equation}
Then by the coupling, we know that the lemma is also true.

Now we are going to construct the coupling between $(Z_t)$ and
$(X_t,Y_t)$ for $t \le T_1^{XY}$. Let $z_0$ be the initial state
satisfying that for any $i\in[1,\ldots ,b]$, if $x_0(\ell_i) = y_0 (\ell_i)
\in\C_Z$ then $z_0 (v_i) = x_0 (\ell_i)$, otherwise we give an
arbitrary color to the vertex $v_i$. On each step $t$, we first
randomly select a vertex in $G^*$ to update in $(X_t,Y_t)$, and
accordingly, we select the corresponding vertex in $G_Z$ to update in $Z_t$:
\begin{itemize}
\item If the vertex is a leaf $\ell_i$,
$(X_t,Y_t)$ selects a random color $c$ or a disagreement to update. If
$c \in\C_Z$, then we give the same color to $v_i$ in $Z_t$; otherwise
we give a random color to $v_i$.

\item If the vertex is the root $r$,
recolor the root on $(X_t, Y_t)$ according to the maximal one-step
coupling and pick a random color in $\C_Z$ to recolor $v_0$
in~$Z$.\looseness=-1
\end{itemize}

Observe that $A_{Z_t} \subseteq A_{X_t}(r) \cap A_{Y_t}(r)$ for any
$0 \le t \le T_1^{XY}$, which implies that $T_1^{Z} \ge T_1^{XY}$ holds
with probability $1$. Now we will show that \eqref{eqn:tz} holds.
Let $\E$ be the event that, in $(Z_t)$, every vertex in the graph $G_Z$
will be recolored at least once within the first $2(b+1)\ln b$ steps.
Let $t_z$ be the first time after time $2(b+1)\ln b$ when the dynamics
$(Z_t)$ recolors the root.
For each color $c \in\C_Z$, define the indicator function $\mathbf
{1}_c := \mathbf{1}\{c \neq Z_{t_z} (v_i), \forall1 \le i \le b \}$.
These indicator functions are negatively associated to each other
(cf. Theorem 14 in~\cite{DR}). It follows by elementary calculation
that, conditioned on $t_z=t$ for some $t>2(b+1)\ln b$ and for large
enough $b$,
we have
%
%
\begin{eqnarray}\label{eq:NA-bound}
&&
\Probability[{A_{Z_t} \neq\varnothing} \mid{t_z = t} ]
\nonumber\\
&&\qquad\ge
\Probability[{\E}] \cdot\Probability[{A_{Z_t} \neq\varnothing}
\mid{t_z = t,\E} ]\nonumber\\
&&\qquad\ge
0.99\Probability[{A_{Z_t} \neq\varnothing} \mid{t_z = t,\E} ]
\qquad\mbox{(since
$\Probability[{\E}] > 1-1/b^2$)}
\nonumber
\\[-8pt]
\\[-8pt]
\nonumber
&&\qquad\ge
0.99 \biggl(1 - \prod_{c\in\C_z}\Probability[{\mathbf{1}_c = 0} \mid
{t_z = t,\E } ]\biggr) \qquad\mbox{(negative association)}\\
&&\qquad\ge
0.99\biggl(1 - \biggl(1 - \biggl(1-\frac{1}{|\C_Z|}\biggr)^b
\biggr)^{|\C
_Z|}\biggr)\nonumber
\\
&&\qquad \ge \frac{1}{3(1+\eps)b^\eps\ln b}.\nonumber
\end{eqnarray}
Since $\Probability[{t_z \le4(b+1)\ln b}]> 1-1/b^2$, by applying
\eqref
{eq:NA-bound}, we have
\begin{eqnarray*}
\Probability[{T_1^{Z} < 4(b+1)\ln b} \mid{Z_0=z_0} ]
&\ge&\sum_{t=2(b+1)\ln b}^{4(b+1)\ln b} \Probability[{A_{Z_t} \neq
\varnothing } \mid{t_z = t} ] \cdot\Probability[{t_z = t}]\\
&\ge&\frac{\Probability[{t_z \le4(b+1)\ln b}]}{3(1+\eps)b^\eps
\ln b}\\
&\ge&\frac{1}{4(1+\eps)b^\eps\ln b}.
\end{eqnarray*}
This completes the proof of Lemma~\ref{lem:common-avail}.
\end{pf}

We also know that when the root is recolored, if $|A_X(r)\oplus A_Y(r)|
\le2$ and $|A_X(r)\cap A_Y(r)|\ge1$ holds, then the probability
that the root will be recolored to the same color in both $X$ and $Y$
is at least $1/2$.
Hence, at time $T_1 = T_1^{XY}$, with probability at least $1/2$,
the root will become an agreement. Combining with Lemma \ref
{lem:warmup}, we prove that with probability at least $1/O((1+\eps
)b^\eps\ln b)$ when $\eps> 0$, starting from arbitrary initial states
$(x_0,y_0)$, the root will couple in at most $12(b+1)\ln b$ steps and
by that time all the disagreements (if there is any) in the leaves are
of the same type. When $\eps= 0$, combining with Lemma \ref
{lem:warmup0}, we get that the probability of the same event happening
is at least $1/O(\ln^3 b)$.

The last step is to let all of the disagreements in the leaves go away
without changing the root to a disagreement, again with constant probability,
after $\Tlev= 4 (b+1)\ln{b}$ more steps. Here is the precise statement
of the lemma.

\begin{lemma}
\label{lem:leafcoupling}
For $\eps\geq0$, consider
a pair of initial states $(x_0,y_0)$ where
the root $r$ agrees [i.e., $x_0(r) = y_0(r)]$
and all of the leaf disagreements are of the same type
[i.e., there is a pair of colors ${c_1},\cb$ such that
for all $\ell\in D^L_0$, we have $x_0(\ell) = {c_1}$ and $y_0(\ell)
= \cb$].
Then, with probability at least $1/2$ after $\Tlev= 4(b+1)\ln b$
steps, we have $X_{\Tlev} = Y_{\Tlev}$.
\end{lemma}

\begin{pf}
First, observe that with high probability after $\Tlev$ steps, all of
the leaves will be recolored at least once.\vadjust{\goodbreak}
Assuming all of the leaves are recolored at least once,
if the root does not become a disagreement within these $\Tlev$ steps,
then all of the leaves will be agreements.
Therefore, we just need to show that the root will not change to a
disagreement in $\Tlev$ steps with probability at least $3/5$. This is
done by a coupling argument.

Let $t_2$ be the first time when the root becomes a disagreement, that
is, $X_{t_2}(r) \neq Y_{t_2}(r)$.
Note, since any disagreements on the
leaves are colored $c_1$ in $X_0$ and $c_2$ in $Y_0$,
either $X_{t_2}(r) = \cb$ and/or $Y_{t_2}(r) = {c_1}$. Therefore, we
define the stopping times $T^X_2$ and $T^Y_2$ as follows:
\[
T^X_2 = \min\{t\dvtx X_t(r) = \cb\},\qquad T^Y_2 = \min\{t\dvtx Y_t(r) =
{c_1}\}.
\]
%
We can assume without loss of generality that $X_0(r)$ [and hence $Y_0(r)$]
does not equal either ${c_1}$ or $\cb$.
Otherwise, by the hypothesis of the lemma, there are no disagreements
in the leaves, and hence $X_0 = Y_0$. Hence, our goal is to show that
\[
\Probability[{T^X_2 \le\Tlev\mbox{ or } T^Y_2 \le\Tlev}] <
\tfrac{2}{5}.
\]
And the main step is to show that
%
%
\begin{equation}
\label{eqn:c1}
\Probability[{T^X_2 \le\Tlev}] < \tfrac{1}{5}.
\end{equation}

Let $(S_t)$ be a random subset process on $V(G^*)$. Each time it picks
a vertex~$v$:
\begin{itemize}
\item if $v \neq r$, with probability $1/(k-1)$, $S_{t+1} = S_t \cup\{
v\}$ and with probability $1- 1/(k-1)$, $S_{t+1} = S_t \setminus\{v\}$;
\item if $v = r$, if $S_t = \varnothing$, then $S_{t+1} = \{r\}$,
otherwise $S_{t+1} = S_t$.
\end{itemize}

Let us define $
T^S = \min_t \{t\dvtx r \in S_t\}
$.
We are going to couple $(S_t)$ with $(X_t)$ such that $\{v \in V(G^*)
\dvtx
X_t(v) = \cb\} \subseteq S_t$. This implies
$T^S \le T^X_2$. And if we can show that
$
\Probability[{T^S \le\Tlev}] \le1/5,
$
then we have proved inequality \eqref{eqn:c1}.

The coupling $(X_t,S_t)$ is defined as follows. We start with $S_0 =
X_0^{-1}(\cb)$, the set of vertices of color $\cb$ in the initial coloring.
Each time both processes picks the same vertex $v$ to update.
\begin{itemize}
\item If $v = r$, $X_t$ and $S_t$ act independently at this time.

\item If $v \neq r$ and $X_t(r) \neq\cb$, then $X_t$ chooses a random
color different from the root to recolor $v$, and if that color is not
$\cb$, $S_{t+1} = S_{t} \setminus\{v\}$ otherwise $S_{t+1} = S_{t}
\cup\{v\}$.

\item If $v \neq r$ and $X_t(r) = \cb$, then $X_t$ chooses a random
color different from $\cb$ to recolor $v$, and if that color is not ${c_1}
$, $S_{t+1} = S_{t} \setminus\{v\}$, otherwise $S_{t+1} = S_{t} \cup\{
v\}$.
\end{itemize}
It is easy to see that this is a valid coupling. More importantly, it
satisfies $X_t^{-1}(\cb) \subseteq S_t$.

Now we are going to show that $\Probability[{T^S \le\Tlev}] < 1/5$
holds. It is
not hard to show that with probability at least $0.9$,
the first time when the root is updated is later than $0.1b$ steps. We
now condition on this event. The indicators of whether each leaf is in
$S_t$ or not during those $0.1b$ steps are negatively associated (cf.
Theorem~14 in~\cite{DR}). Then by using the Chernoff bound with
negative association among the random variables (cf. Proposition 7 in
\cite{DR}), it can be shown that with high probability at least $\geq
0.01b$ many different leaves are recolored before the first time we
recolor the root. Thus, together with the proof of Proposition \ref
{prop:roothit}, 
we can claim that with probability at least $0.85$, before the first
$t$ such that $r \in S_t$, at least $0.01b$ many leaves have been
recolored, and root will be recolored at most $20\ln b$ times
before~$\Tlev$. Denote this event as $\E$. We have
\[
\Probability[{T^S \le\Tlev}] \le\Probability[{T^S \le\Tlev}
\mid{\E} ] + \Probability[{\bar{\E}}]
\le\Probability[{T^S \le\Tlev} \mid{\E} ] + 0.15.
\]
In fact $\Probability[{T^S \le\Tlev} \mid{\E} ]$ can be
arbitrarily small when
$b$ grows, since at each time~$t$ we update the root in $(S_t)$, we
know that the probability of $S_{t-1} = \varnothing$ is at most
$b^{-0.01(1+\eps)}$,
and we know that the root updates at most $20\ln b$ times.

In conclusion, we proved inequality \eqref{eqn:c1} and hence the lemma.
\end{pf}
Finally, by combining Lemmas~\ref{lem:warmup},~\ref{lem:common-avail}
and~\ref{lem:leafcoupling} together, we can conclude that: when $\eps>
0$, with probability at least $1/(20(1+\eps)b^\eps\ln b)$ after $t =
\twarm+ T_1 +  \Tlev< T$ steps of the coupling, we have $X_t = Y_t$;
when $\eps= 0$,
from Lemmas~\ref{lem:warmup0},~\ref{lem:common-avail} and~\ref
{lem:leafcoupling},
we have that
with probability at least $1/(20 \ln^3 b)$ after $t = \twarmp+ T_1 +
\Tlev< T$ steps of the coupling, we have $X_t = Y_t$, which proves
Lemma~\ref{lem:upper-below}.

\section{\texorpdfstring{Upper bound on mixing time for $C>1$: Proof of Theorem \protect\ref{thm:star-upper-above}}
{Upper bound on mixing time for $C>1$: Proof of Theorem 5}}
\label{upabove}
In this section we analyze the upper bound of the mixing time of the
Glauber dynamics
on the star graph $G^*$ when $k = Cb/\ln{b}$ for $C>1$. Here, let
\[
\epsu:= C - 1,
\]
and hence, $k = (1+\epsu)b/\ln b$.

We will analyze the maximal one-step coupling using a weighted
Hamming distance. The root $r$ will have weight $w(r)=b^{\epsu/2} > 1$
and the leaves will have weight $w(v)=1$. For a set of
vertices $S$, let $w(S) = \sum_{v\in S} w(v)$.
Let $D^r_t$ denote whether there is a disagreement at the root.

We want to show that the coupling decreases the distance
in expectation.
Hence, we say a pair of colorings $(X_0,Y_0)$ are $\eta$-distance-decreasing
if there exists a coupling $(X_0,Y_0)\rightarrow(X_1,Y_1)$
such that
\[
{\Expectation[{w(D_{1})} \mid{X_{0},Y_{0} } ]}<(1-\eta)w(D_{0}).
\]
To simplify the analysis of the
coupling, we will use the following theorem
of Hayes and Vigoda~\cite{HV} to utilize properties of the stationary
distribution.
The quantity $\operatorname{diam}(\Omega)$ is the diameter of $\Omega
$ with respect
to the Glauber dynamics. In our case, a trivial bound is
$\operatorname{diam}(\Omega)\le2b$.

\begin{theorem}[{(\cite{HV}, Theorem 1.2)}]
\label{thm:stationarity}
Let $\eta>0$. Suppose $S\subseteq\Omega$ such that every
$(X_0,Y_0)\in S\times\Omega$ is $\eta$-distance-decreasing, and
\[
\pi(S)\ge1-\frac{\eta}{16 \operatorname{diam}(\Omega)},
\]
then the mixing time is
\[
\Tmix\le
3\eta^{-1}\lceil\ln(32 \operatorname{diam}(\Omega))\rceil.
\]
\end{theorem}

We use $S$ as the set of colorings where the root has many available
colors. 
Along the lines of the
Dyer--Frieze~\cite{DyerFrieze} local uniformity results,
we will prove the following statement about the available colors for
the root $r$
in a random coloring.
\begin{lemma}\label{lem:avail-above}
Let $X$ be a random coloring of the star graph on $b$ vertices.
For every $\epsu>0$,
there exists $b_{0}$, such that for all $b>b_{0}$ and $k = (1+\epsu
)b/\ln b$,
\[
\Probability[{|{{A_{X}(r)}}| > b^{0.9\epsu}}]>1-\exp(-b^{0.99\epsu}/10).
\]
\end{lemma}

Hence, we let the set $S$ be those colorings $X\in\Omega$ where
$|{{A_{X}(r)}}|\ge b^{0.9\epsu}$.

\subsection{Analyzing the coupling}

We need to analyze ${\Expectation[{w(D_{1})} \mid{X_{0},Y_{0}} ]} $.
Note, when a leaf $v$ is recolored,
if the root is a disagreement [i.e., $X_0(r)\neq Y_0(r)$], then
with probability $1/(k-1)$ we have $X_{1}(v)\neq Y_{1}(v)$.
Hence,
\begin{eqnarray*}
&&{\Expectation[{w(D^L_{1})} \mid{X_{0},Y_{0}} ]}\\
&&\qquad=\sum_{v\in V\setminus\{r\}} w(v) \bigl[\Probability[{v \mbox{ is
recolored}}]\\
&&\hspace*{56pt}\qquad\quad{}\cdot\Probability[{X_1(v)\neq Y_1(v)} \mid{v \mbox{
is recolored, } X_0,Y_0} ]\\
&&\hspace*{56pt}\qquad\quad{} + (1 - \Probability[{v \mbox{ is recolored}}]){\mathbf{1}
[{X_0(v)\neq Y_0(v)}]}\bigr]\\
&&\qquad= \frac{b}{b+1} \frac{{\mathbf{1} [{r\in D_0}]}}{k-1} + \biggl(1-\frac
{1}{b+1}\biggr)w(D^L_0).
\end{eqnarray*}
There is probability at most $|D^L_0|/\max\{|A_{X_0}(r)|,|A_{Y_0}(r)|\}
$ that
$X_{1}(r)\neq Y_{1}(r)$, when the root $r$ is recolored.
Hence, for $X_0\in S$, we have
\begin{eqnarray*}
&&{\Expectation[{w(D_{1}^r)} \mid{X_{0},Y_{0}} ]}\\
&&\qquad\leq w(r)\frac{1}{b+1}\frac{|D^L_0|}{\max\{
|A_{X_0}(r)|,|A_{Y_0}(r)|\}}
+ \biggl(1-\frac{1}{b+1}\biggr)w(D_0^r)\\
&&\qquad\leq\frac{|D^L_0|b^{-\epsu/3}}{b+1}
+ \biggl(1-\frac{1}{b+1}\biggr)w(D_0^r).
\end{eqnarray*}

Therefore, for $(X_0,Y_0)\in S\times\Omega$, we have
\begin{eqnarray*}
&&{\Expectation[{w(D_{1})} \mid{X_{0},Y_{0}} ]}\\
&&\qquad \leq
\frac{1}{b+1}\biggl( {\mathbf{1} [{r\in D_{0}}]}\frac{b}{k-1}+b^{-\epsu
/3}|D_{0}^{L}|\biggr)
+\biggl(1-\frac{1}{b+1}\biggr)w(D_{0}) \\
&&\qquad \leq
w(D_0) + \frac{1}{b+1}\bigl( -w(D_0) + {\mathbf{1} [{r\in
D_{0}}]}w(r)b^{-\epsu
/3}+b^{-\epsu/3}|D_{0}^{L}|\bigr)\\
&&\qquad \leq
w(D_0) + \frac{1}{b}( -1 + b^{-\epsu/4})w(D_0).
\end{eqnarray*}
Thus, they are $\eta$-distance-decreasing for $\eta=(1-b^{-\epsu/4})/b$.

Now applying Theorem~\ref{thm:stationarity}, by Lemma~\ref{lem:avail-above}
we have the necessary bound on $\pi(S)$, and thus conclude,
for $b$ sufficiently large, we have
\[
\Tmix\le(6b\ln{b})/(1-b^{-\epsu/4}) \le12b\ln{b}.
\]

This completes the proof of Theorem~\ref{thm:star-upper-above},
except for the proof of Lemma~\ref{lem:avail-above}.

\begin{pf*}{Proof of Lemma \protect\ref{lem:avail-above}}
Fix the color of the root to be $c$.
Let $\sigma$ be a random coloring conditional on the root
receiving color $c$.
We are going to prove that
\[
\Probability[{|A_{\sigma}(r)|\le b^{0.9\epsu}} \mid{\sigma(r)=c} ]
<\exp(-b^{0.99\epsu}/10).
\]

For each color $i\in C\setminus\{c\}$, let $Z_{i}$
be the indicator function that $c\in A_{\sigma}(r)$.
$|A_{\sigma}(r)|=\sum_{i\in C}Z_{i}$.
By Theorem 14 in~\cite{DR},
the $Z_{i}$'s are negatively associated with each other once the root
is fixed.
Note that for $b$ sufficiently large,
\begin{eqnarray*}
{\Expectation[{|A_{\sigma}(r)|}]}
& \ge& k\exp\bigl(-b/({k-1})\bigr)
\\
&\ge& b^{0.99\epsu}.
\end{eqnarray*}
Now applying the Chernoff bound, which holds for negatively
associated random variables (cf. Proposition 7 in~\cite{DR}),
we have
\[
\Probability[{|A_{\sigma}(r)|\le b^{0.9\epsu}} \mid{\sigma(r)=c} ]
<\exp(-b^{0.99\epsu}/10).
\]
\upqed\end{pf*}

\section{\texorpdfstring{Proof of the lower bounds below the threshold in Theorem \protect\ref{main-theorem}}
{Proof of the lower bounds below the threshold in Theorem 1}}
\label{lbbelow}
In this section we prove that when $C < 1$,
%
%
\begin{equation}
\Trel= \Omega\bigl(n^{1/C - o(1)}\bigr).
\end{equation}
In the remainder of this section, 
let $L(T)$, or simply $L$, denote the leaves of $T$, and the root is
denoted by $r$.
For a vertex $v$ of $T$, let $T_v$ denote the subtree of $T$ rooted at
$v$, and $T^*_v$ denote $T_v \setminus\{v\}$. For convenience,
in this section, let $
\epsl:= 1/C - 1
$, and hence $k=b/(1+\epsl)\ln{b}$.

In coloring $\sigma\in\Omega(T)$, we say a vertex $v$
is \textit{frozen} in $\sigma$
if, in the subtree $T_v$, the coloring $\sigma(L(T_v))$ of the leaves
of $T_v$
forces the color for $v$. In other words, $v$ is frozen in $\sigma$ if
for all $\eta\in\Omega$ where $\eta(L(T_v))=\sigma(L(T_v))$,
we have $\eta(v)=\sigma(v)$. Note, by definition, the leaves are always
frozen. Observe that for a vertex to be frozen, its frozen children
must ``block''
all other color choices. This is formalized in the following observation
as in~\cite{GJK}.

\begin{observation}
A vertex $v$ where $h(v)>0$
is frozen in coloring $\sigma$ if and only if, for every color $c\neq
\sigma(v)$,
there is a child $w$ of $v$ where $\sigma(w)=c$ and $w$ is frozen.
\end{observation}

Using this inductional way of defining a vertex being ``frozen'' in a
coloring, we can further show the following lemma. It is a
generalization of
Lemma $8$ in~\cite{GJK}, whichis only applied to the case $\epsl\ge
1$, that is, $C \le1/2$.

\begin{lemma}
\label{lem:frozen}
For any $\epsl\in(0,1)$,
in a random coloring of tree $T$, the probability that a vertex of $T$
is not frozen is at most $b^{-\epsl}$. For the leaves in $T$, by
definition, they are always frozen.
\end{lemma}

\subsection{Upper bound on the conductance}
Let $S_c=S_c(T)$ denote those colorings in $\Omega(T)$ where the root
of $T$
is frozen to color $c$. Let $S=\bigcup_{1\le c\le k/2} S_c$. We will
analyze the conductance
of $S$ to lower bound the mixing time.

To upper bound the conductance of $S$, we need to bound the number
of colorings $\sigma\in S$ which can leave $S$ with one transition,
and also the total number of transitions leaving $S$.
To unfreeze the root, we need to recolor a leaf.
Thus, we need to bound the number of colorings frozen at the root
which can become unfrozen
by one recoloring, and in that case,
we need to bound the number of
leaves which can be recolored to unfreeze the root.
For a coloring $\sigma$, vertex $v$ and color $c$, let
$\sigma^{v\rightarrow c}$ denote the coloring obtained by
recoloring $v$ to $c$.

We capture the colorings on the ``frontier'' of $S$ as follows.
For tree $T$, coloring $\sigma\in\Omega(T)$, a vertex $v$ and a leaf
$z$ of $T_v$,
let $\CE{\sigma}{v,z}$ denote the event that
the coloring~$\sigma$ is frozen at the vertex $v$ of $T$
and there exists a color $c$ where
the coloring $\sigma^{z\rightarrow c}$ is not frozen at the vertex $v$.
By definition, this event only depends on the configurations at the
leaves of the subtree $T_v$.
In particular, for the root of the tree, let $\E(\sigma,z):=\CE
{\sigma
}{r,z}$ and $\mathbf{1}_{\sigma,z}$ be the indicator of it.

We can convert the above intuition
into the following upper bound on conductance of $S$ (similarly to
Lemma 10 in~\cite{GJK}).

\begin{lemma}
\label{lem:conductance}
\[
\Phi_S
\le
\frac{6}{n}\sum_{z\in L(T)} \Probability_{\sigma\in\Omega}[{\E
(\sigma,z)}].
\label{eqn:cond1}
\]
\end{lemma}

Now if we can prove that
%
%
\begin{equation}
\label{eq:main-lower}
\Probability_{\sigma\in\Omega}[{\E(\sigma,z)}]\le b^{-(1+\epsl-o(1))H},
\end{equation}
where $o(1)$ is an inverse polynomial of $b$ when $\epsl< 1$ and
equals to zero when $\epsl\ge1$. This will be clarified later in the
proof of Lemma~\ref{lem:Abound}.
Then by plugging this back into the upper bound \eqref{eqn:cond1}, we
get
\[
\Phi_S\le\frac{6}{n}\cdot b^H \cdot b^{-(1+\epsl-o(1))H} \le
20n^{-1-\epsl+o(1)}.
\]
Therefore, we can conclude that the conductance of this Glauber
dynamics is $O(n^{-1-\epsl+o(1)})$, and hence by \eqref{eqn:TmixTrel}
and \eqref{eqn:lbound-Trel}, the mixing time and the relaxation time is
$\Omega(n^{1/C-o(1)})$.

\subsection{\texorpdfstring{Proof of (\protect\ref{eq:main-lower})}{Proof of (11)}}
Let $\Omega^*=\{\sigma\in\Omega\dvtx\sigma(r)=c^*\}$ be the set of
colors where the root is colored $c^*$.
By symmetry, it is easy to see that $
\Probability_{\sigma\in\Omega}[{\E(\sigma,z)}]=\Probability
_{\sigma\in \Omega^*}[{\E (\sigma,z)}]
$. Therefore, for the remainder of the proof we condition on the root
being colored $c^*$.
To simplify the notation, we denote $B:= b^{-(1+\epsl-o(1))}$.
Let $w_0=r, w_1, \ldots, w_{H-1}, w_H=z$
denote the path in $T$ from the root $r$ down to the leaf $z$. We will
show by induction that
\begin{eqnarray*}
\Probability_{\sigma\in\Omega^*}[{\CE{\sigma}{r,z}}]
& \le&
B \Probability_{\sigma\in\Omega^*}[{\CE{\sigma}{w_1,z}}]
\le B^2 \Probability_{\sigma\in\Omega^*}[{\CE{\sigma}{w_2,z}}]
\\ &
\le&\cdots\le B^H \Probability_{\sigma\in\Omega^*}[{\CE{\sigma
}{w_H,z}}] = B^H.
\end{eqnarray*}
For the event $\E(\sigma,z)$ to occur, we need that along the path from
the leaf $z$ to the root $r$, unfreezing each of these vertices will
``free'' a color for their parent. More precisely,
for $\sigma$ to be in $\E(\sigma,z)$, $w_1$ has to be frozen because
the color of $z$ only affects the root through $w_1$, and if
$w_1$ is not frozen, then it cannot affect the root becoming unfrozen.
In order for the root to become unfrozen by changing the
color of the leaf $z$, it must also occur that $w_1$ becomes unfrozen at
the same time, hence $\sigma\in\CE{\sigma}{w_1,z}$, that is, $\E
(\sigma,z)
\subseteq\CE{\sigma}{w_1,z}$
and more generally, $\CE{\sigma}{w_i,z} \subseteq\CE{\sigma
}{w_{i+1},z}$.

For each $1\le i\le H$, let
$\CA{\sigma}{w_i,z}$
denote the event that no sibling $y$ of $w_i$ satisfies both of the following:
$\sigma(y)=\sigma(w_i)$ and $\sigma$ is frozen at $y$.
By the siblings of $w_i$, as usual we mean the children (other than
$w_i$) of $w_{i-1}$.
The event $\CE{\sigma}{w_i,z}$ implies the fact that $w_{i+1}$ is the
only child that causes $w_i$ simultaneously being frozen and being
blocked from using color $\sigma(w_{i+1})$, which means $\CE{\sigma
}{w_{i},z} \subseteq\CA{\sigma}{w_{i+1},z}$. We will show the
following lemma for bounding the probability of $\CA{\sigma}{w_1,z}$.
\begin{lemma}
\label{lem:Abound}
Let $\C^*=\C-c^*$. For a fixed color $c_1 \in\C^*$,
\[
\Probability_{\sigma\in\Omega^*}[{\CA{\sigma}{w_1,z}} \mid
{\sigma (w_1)=c_1} ]\le
B = b^{-(1+\epsl-o(1))}.
\]
%
\end{lemma}

Observe that the events $\CA{\sigma}{1,z}$ and $\CE{\sigma}{w_1,z}$ are
independent, conditioned on the fixed colors of the root and $w_1$,
because they depend on the configurations of different parts of leaves.
Then we have that for each $c_1\in\C^*$,
%
%
\begin{eqnarray}\label{eq:fixcolor}
&&
\Probability_{\sigma\in\Omega^*}\bigl[{\bigl(\sigma(w_1)=c_1\bigr)\cap\CE
{\sigma }{w_1,z}\cap\CA{\sigma}{1,z}}\bigr]
\nonumber\\
&&\qquad=
\Probability_{\sigma\in\Omega^*}[{\CE{\sigma}{w_1,z}} \mid
{\sigma (w_1)=c_1} ]\cdot
\Probability_{\sigma\in\Omega^*}[{\CA{\sigma}{1,z}} \mid
{\sigma (w_1)=c_1} ]\cdot
\frac{1}{k-1}\\
&&\qquad\le
\frac{B^{H-1}\cdot B}{k-1},\nonumber
\end{eqnarray}
where the last inequality is by the inductive hypothesis applied on the
complete tree $T_{w_1}$ of height $H-1$ and Lemma~\ref{lem:Abound}.

Finally, by the fact that $\CE{\sigma}{r,z} \subseteq\CE{\sigma
}{w_1,z}\cap\CA{\sigma}{1,z}$ and \eqref{eq:fixcolor} above, we have
\begin{eqnarray*}
\Probability_{\sigma\in\Omega^*}[{\CE{\sigma}{r,z}}]
&\le&
\Probability_{\sigma\in\Omega^*}[{\CE{\sigma}{w_1,z}\cap\CA
{\sigma }{w_1,z}}]\\
&=&
\sum_{c_1\in\C^*} \Probability_{\sigma\in\Omega^*}\bigl[{\bigl(\sigma
(w_1)=c_1\bigr)\cap \CE{\sigma}{w_1,z}\cap\CA{\sigma}{w_1,z}}\bigr]\\
&\le& B^H.
\end{eqnarray*}
This completes the proof of \eqref{eq:main-lower}. To complete the
proof of the lower bounds when $C<1$ in Theorem~\ref{main-theorem}, we
need to prove
Lemmas~\ref{lem:frozen},~\ref{lem:conductance} and~\ref{lem:Abound}.

\subsection{Proofs of lemmas}
\mbox{}
\begin{pf*}{Proof of Lemma \protect\ref{lem:frozen}}
The proof is very similar to the proof of Lemma 8 in~\cite{GJK}. We
include it here for completeness.

Let $U_\ell$ be the probability that a vertex at the height $\ell$ is
not frozen. We are going to prove that $U_\ell< b^{-\epsl}$ by induction.

First of all, by definition, $U_0 = 0$ since they are leaves.
Let $v$ be a vertex at height $\ell> 0$. Since the probability that
the color of $v$ equals $c$ is independent from the probability that
$v$ is frozen, therefore we can just fix the color of $v$ to some
$c^*\in\C$, and hence
\[
U_\ell=\Probability[{v \mbox{ is not frozen in }\sigma} \mid
{\sigma(v)=c^*} ].
\]
Let $w$ be a child of $v$. Again by the same argument using the
independency, the probability that $w$ is frozen to color $c$ equals
$\frac{1-U_{\ell-1}}{k-1}$. Thus, the probability that all the children
of $v$ are either not frozen or not colored by using $c$ is
$(1-(1-U_{\ell-1})/({k-1}))^b$.

By the union bound and induction, $U_\ell$
is bounded by
\[
(k-1)\biggl(1-\frac{1-U_{\ell-1}}{k-1}\biggr)^b
\le
(k-1)\exp\biggl(-\frac{b(1-b^{-\epsl})}{k-1}\biggr)
\le
b^{-\epsl},
\]
where the last inequality holds for large $b$.\vadjust{\goodbreak}
\end{pf*}

\begin{pf*}{Proof of Lemma \protect\ref{lem:conductance}}
Let $F:=\bigcup_{c\in\C} S_c$ be the set of colorings that freeze
the root.
As we discussed before, by symmetry, $\pi(S_{c_1})=\pi(S_{c_2})$ for
$c_1,c_2 \in\C$. Then $\pi(\bar{S})
\ge1/2$. Also, by Lemma~\ref{lem:frozen}, we know that
$\pi(F) \ge1-b^{-\epsl}$.
Therefore for any $\epsl> 0$, there exists a $b_0$ such that for all
$b>b_0$, $\pi(S)\pi(\bar{S})\ge1/6$.
From the definition of $\Phi_S$, we know that
\[
\Phi_S\le6\biggl(\sum_{\sigma\in S}\sum_{\eta\in\bar{S}}\pi
(\sigma
)P(\sigma,\eta)\biggr).
\]

Notice that, for any $\sigma\in S_{c_1}$, $\eta\in S_{c_2}$ and
$c_1\neq c_2$, we have $P(\sigma,\eta)=0$,
because
it is impossible to change the color of the frozen root by just one
move. Further, in order to unfreeze the root in one step, the Glauber
dynamics has to first recolor a leaf
and change the color of the leaf so as to unfreeze the root. That is,
$\eta$ can only be $\sigma^{z\rightarrow c}$ for some $z\in L(T)$ and
$c\in\C^*$, where $\C^* = \C-\{\mbox{the color of the parent of $z$
in }\sigma\}$. Therefore,
%
%
\begin{eqnarray}\label{eqn:GD2}
\Phi_S
&\le&
6\sum_{\sigma\in F}\sum_{\eta\in\bar{F}}\pi(\sigma)P(\sigma
,\eta)
\nonumber
\\[-8pt]
\\[-8pt]
\nonumber
&\le&
6\sum_{\sigma\in F}\sum_{z\in L(T)}(\mathbf{1}_{\sigma,z}\pi
(\sigma
)P(\sigma,\eta)),
\end{eqnarray}
where $\mathbf{1}_{\sigma,z}$ is the indicator for the event that
the root in coloring $\sigma$ is frozen, and there exists a color $c$ where
the root in the coloring $\sigma^{z\rightarrow c}$ is not frozen.

By the definition of the Glauber dynamics,
we know that $\pi(\sigma)=1/|\Omega|$ and $P(\sigma,\eta)=1/(n(k-1))$
for the case that the change of color happens at a leaf. Therefore,
from \eqref{eqn:GD2}, we have
\[
\Phi_S
\le\frac{6}{n}\sum_{\sigma\in\Omega}\sum_{z\in L(T)}\frac
{\mathbf
{1}_{\sigma,z}}{|\Omega|}
=\frac{6}{n}\sum_{z\in L(T)} \sum_{\sigma\in\Omega(T)} \frac
{\mathbf
{1}_{\sigma,z}}{|\Omega(T)|}.
\]
\upqed\end{pf*}

\begin{pf*}{Proof of Lemma \protect\ref{lem:Abound}}
When $\epsl< 1$, the probability that all the siblings of $w_1$ are
either not frozen or not colored with $c_1$ is upper bounded by
\begin{eqnarray*}
\biggl(1-\frac{1-U_{H-1}}{k-1}\biggr)^{b-1}
&\le&
\exp\biggl(-\frac{(b-1)(1-b^{-\epsl})}{k-1}\biggr)\\
&\le&
b^{-(1+\epsl)(1 - b^{-\epsl})}.
\end{eqnarray*}
Now we can see that $o(1)$ is actually $(1+\epsl)/b^\epsl$ when
$\epsl
< 1$.

Note that, when $\epsl\ge1$, in the same way it is easy to see that
\[
\biggl(1-\frac{1-U_{H-1}}{k-1}\biggr)^{b-1}
\le b^{-(1+\epsl)}.
\]
\upqed\end{pf*}

\section{\texorpdfstring{A Simple generalization to $k=o(b/\ln{b})$: Proof of Theorem \protect\ref{main-theorem2}}
{A Simple generalization to $k=o(b/\ln{b})$: Proof of Theorem 2}}\label{generalization}

In all of the previous sections, we assumed $k=Cb/\ln b$ where $C$ is
constant. But we are also interested in the case when $k$ is constant,\vadjust{\goodbreak}
say a hundred colors,
and what the mixing time of the Glauber dynamics will be in this case.
Let $\alpha= \alpha(k,b) := b / (k \ln b)$. We would also like to see
how to generalize the upper bound and lower bound analysis assuming
$\alpha$ is any function growing with $b$, that is, when $k$ is
$o(b/\ln b)$.
Actually, all of our proofs will be the same, and we just need to
slightly modify the statements.

For the upper bound, we change Lemma~\ref{lem:upper-below} and Lemma
\ref{lem:common-avail} into the following ones.
\begin{lemma}
\label{lem:upper-below2}
Let $T=20b\ln{b}$. There exists $b_0$, for all $(x_0,y_0)\in\Omega
\times\Omega$,
all $\alpha(k,b) \geq2$ and all $b > b_0$ the following holds:
\[
\Probability[{X_{T} = Y_{T}} \mid{X_0 = x_0, Y_0 = y_0} ] \ge
1/\bigl(20\alpha
(k,b)b^{\alpha(k,b)} \ln b\bigr).
\]
\end{lemma}

\begin{lemma}
\label{lem:common-avail2}
For any pair of initial states $(x_0,y_0)$ where
all of the leaf disagreements are of the same type, then
\[
\Probability[{T_1^{XY} < 4b\ln b} \mid{(X_0,Y_0)=(x_0,y_0)} ] \ge
1/\bigl(4\alpha
(k,b)b^{\alpha(k,b)-1}\ln b\bigr).
\]
\end{lemma}

Then by the same argument as in Section~\ref{upbelow}, we are able to
show that the relaxation time of the Glauber dynamics on $G^*$ is upper
bounded by $O(\alpha b^{\alpha}\ln b)$. Thus, the mixing time of the
Glauber dynamics on the complete tree is bounded by
\[
\Tmix= O\bigl(n^{\alpha+ (\ln{\alpha}+2\ln\ln b + 20)/{\ln b}}\ln
n\bigr),
\]
and the relaxation time is bounded by
\[
\Trel= O\bigl(n^{\alpha+ (\ln{\alpha}+2\ln\ln b + 20)/{\ln b}}\bigr).
\]

For the lower bound, we again place Lemma~\ref{lem:frozen} and Lemma
\ref{lem:Abound} into the following lemmas.
\begin{lemma}
\label{lem:frozen2}
In a random coloring of the tree $T$, the probability that a vertex of
$T$ is not frozen is at most $b^{-1}$.
\end{lemma}

\begin{lemma}
\label{lem:Abound2}
\[
{\Probability_{\sigma\in\Omega^*}[{\CA{\sigma}{w_1,z}} \mid
{\sigma (w_1)=c_1)} ]}\le b^{-\alpha(k,b)}.
\]
\end{lemma}

Then, in exactly the same way as in Section~\ref{lbbelow}, we can show
that the mixing time and the relaxation time of the Glauber dynamics on
the complete tree $T$ when $\alpha\ge2$ is lower bounded by $\Omega
(n^{\alpha}) = \Omega(n^{b/(k\ln b)})$.

\section{\texorpdfstring{Bounding the log-Sobolev constant: Proof of Theorem \protect\ref{thm:LS}}
{Bounding the log-Sobolev constant: Proof of Theorem 6}}
\label{logsobolev}
In this section we will analyze the log-Sobolev constant $\ls$ of the
heat-bath Glauber dynamics on the complete tree by comparing it with
the spectral gap $\gap$.
For completeness, we prove Theorem~\ref{thm:LS}, which is an
improvement over the proof of Theorem 5.7 in Martinelli, Sinclar and
Weitz~\cite{MSW}. In their paper, they proved it for the case
of the Ising model on the complete tree with a fixed boundary\vadjust{\goodbreak}
condition, although they observed that it holds more generally. For
convenience, we will use the same notation for the complete tree and
its vertices; that is, $T_\ell$ stands for both the complete tree of
height $\ell$ and its vertices $V(T_\ell)$.

Let $B\subseteq A \subseteq T$ be two subsets of the vertices on tree
$T$. Let $\eta\in\Omega$ be a configuration. Let $\mathrm{E}_A^{\eta}(f)$ be
the expectation of $f$ under a prefixed distribution $\mu$ in the
region $A$ with boundary condition $\eta$. That is,
\[
\mathrm{E}_A^{\eta}(f) = \sum_{\sigma} \frac{\mu(\sigma)}{Z} f(\sigma),
\]
where $\sigma$ ranges over the configurations that are the same as
$\eta
$ outside $A$ (denoted as $\sigma\sim_A \eta$), and $Z$ is the
normalizing factor.
The quantities $\var_A^{\eta}$ and $\ent_A^{\eta}$ are defined
similarly. If we drop $\eta$, then $\mathrm{E}_A(f), \var_A(f), \ent_A(f)$
become functions from $\Omega$ to $R$. The following are standard facts
concerning variance and entropy: the first is the chain rule, and the
second follows from the so-called tensoring property over a product
distribution; see,
for example,
Proposition 5.6 of~\cite{Ledoux}. In the following, we will use the fact
that the distribution on configurations over the tree with the root
removed, has a product form over the subtrees rooted at the children of
the root,
to satisfy the hypothesis for the tensoring property.

\begin{proposition}
\label{prop:lsdec}
\begin{eqnarray*}
\var_A^{\eta}(f) &=& \mathrm{E}_A^{\eta}(\var_B(f)) + \var_A^{\eta}(\mathrm{E}_B(f)),
\\
\ent_A^{\eta}(f) &=& \mathrm{E}_A^{\eta}(\ent_B(f)) + \ent_A^{\eta}(\mathrm{E}_B(f)).
\end{eqnarray*}
\end{proposition}

\begin{proposition}
\label{prop:lsdec2}
Let $A = \bigcup A_i$ where $A_i$ are disjoint, and suppose that
conditioning on the boundary being $\eta$, the probability of $A_i$'s
being in any configuration for different $i$'s is completely
independent. Then
\[
\var_A^{\eta}(f) \le\sum_i \mathrm{E}_{A}^{\eta}(\var_{A_i}(f))
\]
and
\[
\ent_A^{\eta}(f) \le\sum_i \mathrm{E}_{A}^{\eta}(\ent_{A_i}(f)).
\]
\end{proposition}

\begin{lemma}
\label{lem:ls}
Let $\ls(\ell)$ be the log-Sobolev constant of the heat-bath Glauber
dynamics on the complete tree of height $\ell> 0$ with the root being
attached to an external vertex with a fixed color, then
\[
\ls(\ell)^{-1} \le\ls(\ell-1)^{-1} + \alpha\cdot\gap(\ell)^{-1},
\]
where $\alpha= \frac{\log(k-2)}{1-2/(k-1)} = \ls(0)^{-1}$.
\end{lemma}

\begin{pf}
Let $f$ be any nonnegative function. Let $I$ be the set of vertices in
the complete tree $T_{\ell}$ without the root, that is, $I = T_{\ell}
\setminus\{\operatorname{root}\}$.\vadjust{\goodbreak} Let us first use Proposition \ref
{prop:lsdec} to analyze the $\ent(f)$.
\[
\ent(f) = \mathrm{E}(\ent_I(f)) + \ent(\mathrm{E}_I(f)).
\]

We will bound $\mathrm{E}(\ent_I(f))$ and $\ent(\mathrm{E}_I(f))$ separately. For
$\mathrm{E}(\ent
_I(f))$, by Proposition~\ref{prop:lsdec2}, it can be upper bounded as
%
%
\begin{equation}
\label{eqn:LS1}
\mathrm{E}(\ent_I(f)) \le\sum_{v} \mathrm{E}(\ent_{T_v}(f)),
\end{equation}
where $v$ ranges over all the children of the root of $T_\ell$, and
$T_v$ denotes the subtree of $T_\ell$ rooted at the vertex $v$. Let
$\eta\in\Omega(T_{\ell})$, then for a specific $\ent_{T_v}^{\eta
}(f)$, we then have
%
%
\begin{eqnarray}
\label{eqn:LS2}
\ent_{T_v}^{\eta}(f) \le\ls(\ell-1)^{-1} \D_{T_v}\bigl(\sqrt{f}\bigr),
\end{eqnarray}
where $\D_{T_v}(\sqrt{f})$ is the corresponding Dirichlet form for the
dynamics on the subtree $T_v$. For the heat-bath Glauber dynamics,
since $P(\sigma,\tau)\neq0$ only if they differ at a single vertex, we
can further derive that
%
%
\begin{eqnarray}
\label{eqn:LS3}
\D_{T_v}(f) &=& \frac{1}{2}\sum_{\sigma,\tau} \bigl(f(\sigma) - f(\tau)\bigr)^2
\mu(\sigma) P(\sigma,\tau)
\nonumber
\\[-8pt]
\\[-8pt]
\nonumber
&=& \frac{1}{2}\sum_{x\in T_v} \mathrm{E}_{T_v}^{\eta} \bigl(\var_{\{x\}}(f)\bigr),
\end{eqnarray}
where $\mu(\sigma)$ is the marginal distribution with respect to
$\eta$.

Then, from \eqref{eqn:LS1}, \eqref{eqn:LS2} and the above, we have
\begin{eqnarray*}
\mathrm{E}(\ent_I(f)) &\le& \sum_v \mathrm{E}(\ent_{T_v}(f)) \qquad\mbox{[by \eqref
{eqn:LS1}]}\\
&\le& \sum_v \ls(\ell-1)^{-1} \mathrm{E}\bigl( \D_{T_v}\bigl(\sqrt{f}\bigr) \bigr) \qquad\mbox{[by
\eqref{eqn:LS2}]}\\
& = & \sum_v \ls(\ell-1)^{-1} \mathrm{E}\biggl ( \sum_{x\in T_v}
\mathrm{E}_{T_v}^{\eta}
\bigl[\var_{\{x\}}(f)\bigr] \biggr) \qquad\mbox{[by \eqref{eqn:LS3}]}\\
& = & \ls(\ell-1)^{-1} \sum_{x\in I} \mathrm{E}\bigl( \var_{\{x\}}(f)\bigr)\\
& \le& \ls(\ell-1)^{-1} \D\bigl(\sqrt{f}\bigr)\qquad \mbox{[by applying \eqref
{eqn:LS3} again]}.
\end{eqnarray*}

For $\ent(\mathrm{E}_I(f))$, $\mathrm{E}_I(f)$ can be viewed as a function from $\{
1,2,\ldots,k-1\}$ to $R$ since those $k-1$ values can represent the colors
of the root (boundary). Therefore $\ent(\mathrm{E}_I(f))$ is the entropy of the
random variable $\mathrm{E}_I(f)$ taking $k-1$ values uniformly at random. It is
well known (see, e.g., the Appendix of~\cite{DS}) that $ \frac{\log
(k-2)}{1-2/(k-1)}$ is the inverse of the log-Sobolev constant of the
random walk $\R$ on the complete graph $K_{k-1}$,\vadjust{\goodbreak} which jumps to
stationarity in one step. Thus, letting
$\alpha= \frac{\log(k-2)}{1-2/(k-1)}$, we may upper bound
$\ent(\mathrm{E}_I(f))$ as follows:
\begin{eqnarray*}
&&\ent(\mathrm{E}_I(f)) \\
&&\qquad\le \alpha\D_{\R}\bigl(\sqrt{\mathrm{E}_I(f)}\bigr) \qquad\mbox{(by the
log-Sobolev inequality)}\\
&&\qquad =  \alpha\var_{\R} \bigl(\sqrt{\mathrm{E}_I(f)}\bigr)\qquad \mbox{[for the complete graph
$P(x,y)=\pi_{\R}(y)$]}\\
&&\qquad =  \alpha\var_{T}\bigl (\sqrt{\mathrm{E}_I(f)}\bigr) \\
&& \qquad\le \alpha\bigl(\mathrm{E} [\mathrm{E}_I(f)] - \mathrm{E}^2\bigl(\sqrt{\mathrm{E}_I(f)}\bigr)\bigr) \qquad\mbox{(by
the definition of the variance)}\\
&&\qquad \le \alpha \mathrm{E}\bigl(\sqrt{f}\bigr)^2 - \mathrm{E}^2\bigl(\sqrt{f}\bigr) \qquad\bigl(\mbox{by the
concavity of }\sqrt{x}\bigr) \\
&&\qquad \le \alpha\gap(\ell)^{-1} \D\bigl(\sqrt{f}\bigr) \qquad\mbox{(by the
definition of the spectral gap).}
\end{eqnarray*}

Putting everything together, we prove
\begin{eqnarray*}
\ent(f) &=& \mathrm{E}(\ent_I(f)) + \ent(\mathrm{E}_I(f)) \\
&\le& \ls(\ell-1)^{-1} \D\bigl(\sqrt{f}\bigr) + \alpha\gap(\ell)^{-1} \D
\bigl(\sqrt{f}\bigr),
\end{eqnarray*}
and then by the definition of $\ls$, we get
\[
\ls(\ell)^{-1} \le\ls(\ell-1)^{-1} + \alpha\gap(\ell)^{-1} .
\]
\upqed\end{pf}
\begin{lemma}
\label{lem:gap}
Let $\ls(\ell)$ be the spectral gap of the heat-bath Glauber dynamics
on the complete tree of height $\ell> 0$ with the root being attached
to an external vertex with a fixed color; then
for $\ell> 0$, we have $\gap(\ell) \le\gap(\ell-1)/b$.
\end{lemma}

\begin{pf}
Let $\D_{\ell}(f)$ and $\var_{\ell}(f)$ be the Dirichlet form and the
variance of function $f\dvtx\Omega(T_{\ell})\rightarrow\R$ for the Glauber
dynamics on the complete tree of height~$\ell$ with the root attached
to an external vertex with a fixed color.
Let $P_{\ell}$ denote the probability transition of the dynamics,
and let $\pi_\ell$ denote its unique stationary distribution.

Let $g$ be the eigenfunction such that $\gap(\ell-1) = \D_{\ell
-1}(g)/\var_{\ell-1}(g)$. Now we are going to construct a function
$f\dvtx\Omega(T_{\ell})\rightarrow\R$, such that $\D_{\ell}(f)
\le\D
_{\ell
-1}(g)$ and $\var_{\ell}(f) = \var_{\ell-1}(g)$. Then, since
\[
\gap(\ell) \le\frac{\D_{\ell}(f)}{\var_{\ell}(f)} \le\frac{\D
_{\ell
-1}(g)}{b\cdot\var_{\ell-1}(g)} = \gap(\ell-1),
\]
we prove the lemma.

Let $A\subseteq T_{\ell}$ be the set of nonleaf vertices of $T_{\ell
}$, that is, $A = T_{\ell} \setminus L(T_\ell)$, where $L(T_\ell)$ is
the set of leaves in the tree $T_{\ell}$. There is a natural
correspondence between vertices in $A$ and in $T_{\ell-1}$. The
function $f$ is then defined in the following way: for $\sigma\in
\Omega(T_{\ell})$ and $\sigma' \in\Omega(T_{\ell-1})$, $f(\sigma
) =
g(\sigma')$ if the configuration $\sigma$ agrees with $\sigma'$ on the
subset $A$.

It is straightforward to show that $\var_\ell(f) = \var_{\ell-1}(g)$.
We will show $\D_{\ell}(f) \le\D_{\ell-1}(g)/b$. By definition,
\[
\D_\ell(f) = \sum_{\sigma,\eta\in\Omega(T_\ell)}\pi_\ell
(\sigma)P_\ell
(\sigma,\eta)\bigl(f(\sigma)-f(\eta)\bigr)^2.
\]

For a subset of vertices $S\subset T_\ell$, let
\[
\Omega(S) = \{\sigma'\in[k]^S\dvtx \mbox{ there exists } \sigma
\in
\Omega
(T_\ell) \mbox{ where } \sigma(A)=\sigma'\}.
\]
Let $\sigma', \eta' \in\Omega(A)$ be colorings of the internal vertices.
Let $\phi, \psi\in\Omega(L(T_\ell))$ be colorings of the leaves.
Finally, let $\circ$ be the concatenation operator, thus
$\sigma' \circ\psi= \sigma\in\Omega(T_{\ell})$ where
$\sigma(A) = \sigma'$ and $\sigma({L(T_\ell)}) = \phi$.
Then we can rewrite the Dirichlet form as
\begin{eqnarray*}
\D_\ell(f) &=&
\sum_{\sigma',\eta'\in\Omega(A)}
\sum_{\phi,\psi\in\Omega(L(T_{\ell}))}
\pi_\ell(\sigma'\circ\phi)P_\ell(\sigma' \circ\phi,\eta'
\circ\psi
)\\
&&\phantom{\sum_{\sigma',\eta'\in\Omega(A)}
\sum_{\phi,\psi\in\Omega(L(T_{\ell}))}}{}\times\bigl(f(\sigma'\circ\phi)-f(\eta' \circ\psi)\bigr)^2.
\end{eqnarray*}

According to the definition of the Glauber dynamics, for
configurations $\sigma, \eta\in\Omega(T_\ell)$ which differ at more
than one vertex, we have $P_\ell(\sigma,\eta) = 0$. Let $\oplus$ denote
the symmetric difference.
Now we can rewrite the Dirichlet form as
\begin{eqnarray*}
\D_\ell(f)
&=&
\sum_{v\in A}
\mathop{\sum_{\sigma',\eta'\in\Omega(A)\dvtx}}_{\sigma' \oplus\eta
'=\{v\}}
\sum_{\phi\in\Omega(L(T_\ell))}
\bigl[ \bigl(f(\sigma'\circ\phi)-f(\eta'\circ\phi)\bigr)^2
\pi_\ell(\sigma'\circ\phi)\\
&&\hspace*{150pt}\qquad{}\times  P_\ell(\sigma' \circ\phi,\eta'
\circ\phi
)\bigr]
\\
&&{}+
\sum_{v\in L(T_\ell)} \sum_{\sigma'\in\Omega(A)} \mathop{\sum_{\phi
,\psi\in
\Omega(L(T_\ell))\dvtx}}_{\phi\oplus\psi= \{v\}}
\bigl[\bigl(f(\sigma'\circ\phi)-f(\sigma'\circ\psi)\bigr)^2 \pi_\ell
(\sigma'\circ
\phi)\\
&&\hspace*{183pt}\qquad{}\times P_\ell(\sigma' \circ\phi,\sigma' \circ\psi)\bigr]
\\
&=& \sum_{v\in A} \mathop{\sum_{\sigma',\eta'\in\Omega(A)\dvtx}}_{
\sigma'
\oplus
\eta'=\{v\}}
\biggl[ \bigl(g(\sigma')-g(\eta')\bigr)^2
\sum_{\phi\in\Omega(L(T_\ell))}
\pi_\ell(\sigma'\circ\phi) \\
&&\hspace*{166pt}\qquad{}\times P_\ell(\sigma' \circ\phi,\eta'
\circ\phi
)\biggr],
\end{eqnarray*}
since $g(\sigma'\circ\phi)=g(\sigma'\circ\psi)=f(\sigma')$.

Thus we only need to consider the case when the sole disagreement is at
an internal
vertex. We can further decompose based on whether the disagreement is an
internal vertex of the tree $T_{\ell-1}$, which we denote as $I$,
or a leaf of $T_{\ell-1}$.

For $v\in L(T_{\ell-1})$, the goal is to bound the sum $\sum_{\phi}
\pi
_{\ell}(\sigma'\circ\phi) P_\ell(\sigma' \circ\phi,\eta' \circ
\phi)$
by $\pi_{\ell-1}(\sigma')/(|T_{\ell-1}|(k-1)b)$, that is, $\pi
_{\ell
-1}(\sigma') P_{\ell-1}(\sigma',\eta')/b$. We have the following
observation: Fix the vertex $v$, for each color $c$ such that $\sigma
'\oplus\eta' = \{v\}$ and $\eta'(v) = c$, the quantity $Q(c) := \sum
_{\phi} \pi_{\ell}(\sigma'\circ\phi) P_\ell(\sigma' \circ\phi
,\eta'
\circ\phi)$ are the same, that is, $Q(c) = Q(c')$ for any two colors
$c \neq c'$ because of the symmetry. Therefore, in order to bound
$Q(c)$, it is easier to bound $\sum_{c\neq\sigma'(v)} Q(c)$ by $\pi
_{\ell-1}(\sigma')/(|T_{\ell-1}|b)$. Then, by taking the average over
$k-1$ colors, we are done. It is a straightforward calculation to upper
bound the sum of $Q(c)$:
%
%
\begin{eqnarray}\label{ineq:leaf-disagree}
\sum_{c\neq\sigma'(v)} Q(c)
&=& \pi_{\ell-1}(\sigma')
\sum_{\phi} \frac{\pi_{\ell}(\sigma'\circ\phi)}{\pi_{\ell
-1}(\sigma')}
\sum_{c\neq\sigma'(v)}
\mathbf{1}\{c \in A_{\sigma'\circ\phi}(v)\} P_\ell(\sigma' \circ
\phi
,\eta' \circ\phi)\nonumber\\
&=& \pi_{\ell-1}(\sigma')
\sum_{\phi} \frac{\pi_{\ell}(\sigma'\circ\phi)}{\pi_{\ell
-1}(\sigma')}
\frac{|A_{\sigma'\circ\phi}(v)|-1}{|T_{\ell}||A_{\sigma'\circ
\phi
}(v)|}\\
&\le&
\pi_{\ell-1}(\sigma')\frac{1}{|T_{\ell-1}|b},\nonumber
\end{eqnarray}
where by definition, $A_{\sigma'\circ\phi}(v)$ is the set of available
colors for vertex $v$ in the configuration $\sigma' \circ\phi$.

Recall that $I$ denotes the internal vertices of $T_{\ell-1}$,
that is, $I=V(T_{\ell-1})\setminus L(T_{\ell-1})$.
Similarly, for $v\in I$ we have
\begin{eqnarray}\label{ineq:internal-disagree}
&&\sum_{\sigma',\eta'\in\Omega(A)\dvtx
\sigma' \oplus\eta'=\{v\}
} \biggl[\bigl(g(\sigma')-g(\eta')\bigr)^2
\nonumber\\
&&\hspace*{68pt}\qquad{}\times \sum_{\phi\in\Omega(L(T_{\ell}))} \pi_\ell(\sigma'\circ\phi)
P_\ell
(\sigma' \circ\phi,\eta' \circ\phi)\biggr]
^{\prime}
\\
&&\qquad= \mathop{\sum_{\sigma',\eta' \in\Omega(T_{\ell-1})\dvtx}}_{
\sigma' \oplus\eta'=\{v\}
} \bigl(g(\sigma')-g(\eta')\bigr)^2 \pi_{\ell-1}(\sigma') P_{\ell-1}(\sigma
',\eta')/b.\nonumber
\end{eqnarray}

Combining \eqref{ineq:leaf-disagree} and \eqref{ineq:internal-disagree},
and summing over $v\in T_{\ell-1}$, we have shown that
$\D_\ell(f) \le\D_{\ell-1}(g)/b$, which implies the lemma.
\end{pf}

\begin{pf*}{Proof of Theorem \protect\ref{thm:LS}}
Now we apply Lemma~\ref{lem:ls} inductively, and we get
\[
\ls^{-1} = \ls^{-1}(H) \le\alpha\bigl(1 + \gap^{-1}(1) + \cdots+
\gap
^{-1}(\lfloor\log_b n \rfloor)\bigr).
\]

Then by applying Lemma~\ref{lem:gap} on the spectral gaps, we can
conclude that
\[
\ls^{-1} \le b \alpha\gap^{-1}(H) \le\gap^{-1}\cdot2b\log{k}.
\]
\upqed\end{pf*}

\section{Conclusions}

Recently, Restrepo et al.~\cite{RSVVY} studied the analogous problem
for the
hard-core model which is defined on independent sets weighted by an
activity $\lambda>0$. In contrast to the picture we have shown for
colorings, Martinelli et al.~\cite{MSW:soda} has shown that on the
complete tree with branching factor $b$, the Glauber dynamics has
$O(n\log{n})$ mixing time for all $\lambda$. Thus, there is no slow-down
at the reconstruction threshold. However, Restrepo et al.~\cite{RSVVY}
show that
there is a boundary condition for the complete tree so that the
Glauber dynamics has a slow-down that appears to coincide with the
reconstruction threshold.


%


\printaddresses


\begin{thebibliography}{32}

\bibitem{ACO}
%
\begin{bincollection}[auto:STB|2012/04/24|10:30:10]
\bauthor{\bsnm{Achlioptas},~\bfnm{D.}\binits{D.}} \AND
\bauthor{\bsnm{Coja-Oghlan},~\bfnm{A.}\binits{A.}}
(\byear{2008}).
\btitle{Algorithmic barriers from phase transitions}.
In \bbooktitle{Proceedings of the 49th Annual IEEE Symposium on
Foundations of
Computer Science (FOCS)}
\bpages{793--802}.
\bptok{imsref}%
\end{bincollection}
%
\endbibitem

\bibitem{Aldous}
%
\begin{bincollection}[mr]
\bauthor{\bsnm{Aldous},~\bfnm{David}\binits{D.}}
(\byear{1983}).
\btitle{Random walks on finite groups and rapidly mixing {M}arkov chains}.
In \bbooktitle{Seminar on Probability, {XVII}}.
\bseries{Lecture Notes in Math.}
\bvolume{986}
\bpages{243--297}.
\bpublisher{Springer}, \baddress{Berlin}.
\bid{mr={0770418}}
\bptok{imsref}%
\end{bincollection}
%
\endbibitem

\bibitem{BKMP}
%
\begin{barticle}[mr]
\bauthor{\bsnm{Berger},~\bfnm{Noam}\binits{N.}},
\bauthor{\bsnm{Kenyon},~\bfnm{Claire}\binits{C.}},
\bauthor{\bsnm{Mossel},~\bfnm{Elchanan}\binits{E.}} \AND
\bauthor{\bsnm{Peres},~\bfnm{Yuval}\binits{Y.}}
(\byear{2005}).
\btitle{Glauber dynamics on trees and hyperbolic graphs}.
\bjournal{Probab. Theory Related Fields}
\bvolume{131}
\bpages{311--340}.
\bid{doi={10.1007/s00440-004-0369-4}, issn={0178-8051}, mr={2123248}}
\bptok{imsref}%
\end{barticle}
%
\endbibitem

\bibitem{BVVW}
%
\begin{barticle}[mr]
\bauthor{\bsnm{Bhatnagar},~\bfnm{Nayantara}\binits{N.}},
\bauthor{\bsnm{Vera},~\bfnm{Juan}\binits{J.}},
\bauthor{\bsnm{Vigoda},~\bfnm{Eric}\binits{E.}} \AND
\bauthor{\bsnm{Weitz},~\bfnm{Dror}\binits{D.}}
(\byear{2011}).
\btitle{Reconstruction for colorings on trees}.
\bjournal{SIAM J. Discrete Math.}
\bvolume{25}
\bpages{809--826}.
\bid{doi={10.1137/090755783}, issn={0895-4801}, mr={2817532}}
\bptok{imsref}%
\end{barticle}
%
\endbibitem


\bibitem{DS}
%
\begin{barticle}[mr]
\bauthor{\bsnm{Diaconis},~\bfnm{P.}\binits{P.}} \AND
\bauthor{\bsnm{Saloff-Coste},~\bfnm{L.}\binits{L.}}
(\byear{1996}).
\btitle{Logarithmic {S}obolev inequalities for finite {M}arkov chains}.
\bjournal{Ann. Appl. Probab.}
\bvolume{6}
\bpages{695--750}.
\bid{doi={10.1214/aoap/1034968224}, issn={1050-5164}, mr={1410112}}
\bptok{imsref}%
\end{barticle}
%
\endbibitem

\bibitem{DLP}
%
\begin{barticle}[mr]
\bauthor{\bsnm{Ding},~\bfnm{Jian}\binits{J.}},
\bauthor{\bsnm{Lubetzky},~\bfnm{Eyal}\binits{E.}} \AND
\bauthor{\bsnm{Peres},~\bfnm{Yuval}\binits{Y.}}
(\byear{2010}).
\btitle{Mixing time of critical {I}sing model on trees is polynomial
in the
height}.
\bjournal{Comm. Math. Phys.}
\bvolume{295}
\bpages{161--207}.
\bid{doi={10.1007/s00220-009-0978-y}, issn={0010-3616}, mr={2585995}}
\bptok{imsref}%
\end{barticle}
%
\endbibitem

\bibitem{DR}
%
\begin{barticle}[mr]
\bauthor{\bsnm{Dubhashi},~\bfnm{Devdatt}\binits{D.}} \AND
\bauthor{\bsnm{Ranjan},~\bfnm{Desh}\binits{D.}}
(\byear{1998}).
\btitle{Balls and bins: A study in negative dependence}.
\bjournal{Random Structures Algorithms}
\bvolume{13}
\bpages{99--124}.
\bid{doi={10.1002/(SICI)1098-2418(199809)13:2\&lt;99::AID-RSA1\&gt;3.0.CO;2-M},
issn={1042-9832}, mr={1642566}}
\bptok{imsref}%
\end{barticle}
%
\endbibitem

\bibitem{DyerFrieze}
%
\begin{barticle}[mr]
\bauthor{\bsnm{Dyer},~\bfnm{Martin}\binits{M.}} \AND
\bauthor{\bsnm{Frieze},~\bfnm{Alan}\binits{A.}}
(\byear{2003}).
\btitle{Randomly coloring graphs with lower bounds on girth and maximum
degree}.
\bjournal{Random Structures Algorithms}
\bvolume{23}
\bpages{167--179}.
\bid{doi={10.1002/rsa.10087}, issn={1042-9832}, mr={1995689}}
\bptok{imsref}%
\end{barticle}
%
\endbibitem

\bibitem{DGJM}
%
\begin{barticle}[mr]
\bauthor{\bsnm{Dyer},~\bfnm{Martin}\binits{M.}},
\bauthor{\bsnm{Goldberg},~\bfnm{Leslie~Ann}\binits{L.~A.}},
\bauthor{\bsnm{Jerrum},~\bfnm{Mark}\binits{M.}} \AND
\bauthor{\bsnm{Martin},~\bfnm{Russell}\binits{R.}}
(\byear{2006}).
\btitle{Markov chain comparison}.
\bjournal{Probab. Surv.}
\bvolume{3}
\bpages{89--111}.
\bid{doi={10.1214/154957806000000041}, issn={1549-5787}, mr={2216963}}
\bptok{imsref}%
\end{barticle}
%
\endbibitem

\bibitem{FriezeVigoda}
%
\begin{bincollection}[mr]
\bauthor{\bsnm{Frieze},~\bfnm{Alan}\binits{A.}} \AND
\bauthor{\bsnm{Vigoda},~\bfnm{Eric}\binits{E.}}
(\byear{2007}).
\btitle{A survey on the use of {M}arkov chains to randomly sample colourings}.
In \bbooktitle{Combinatorics, Complexity, and Chance}
(\beditor{\binits{G.} \bsnm{Grimmett}}
\AND
\beditor{\binits{C.}~\bsnm{\mbox{McDiarmid}}}, eds.).
\bseries{Oxford Lecture Ser. Math. Appl.}
\bvolume{34}
\bpages{53--71}.
\bpublisher{Oxford Univ. Press}, \baddress{Oxford}.
\bid{doi={10.1093/acprof:oso/9780198571278.003.0004}, mr={2314561}}
\bptok{imsref}%
\end{bincollection}
%
\endbibitem

\bibitem{GJK}
%
\begin{barticle}[mr]
\bauthor{\bsnm{Goldberg},~\bfnm{Leslie~Ann}\binits{L.~A.}},
\bauthor{\bsnm{Jerrum},~\bfnm{Mark}\binits{M.}} \AND
\bauthor{\bsnm{Karpinski},~\bfnm{Marek}\binits{M.}}
(\byear{2010}).
\btitle{The mixing time of {G}lauber dynamics for coloring regular trees}.
\bjournal{Random Structures Algorithms}
\bvolume{36}
\bpages{464--476}.
\bid{doi={10.1002/rsa.20303}, issn={1042-9832}, mr={2666764}}
\bptok{imsref}%
\end{barticle}
%
\endbibitem

\bibitem{HS}
%
\begin{barticle}[mr]
\bauthor{\bsnm{Hayes},~\bfnm{Thomas~P.}\binits{T.~P.}} \AND
\bauthor{\bsnm{Sinclair},~\bfnm{Alistair}\binits{A.}}
(\byear{2007}).
\btitle{A general lower bound for mixing of single-site dynamics on graphs}.
\bjournal{Ann. Appl. Probab.}
\bvolume{17}
\bpages{931--952}.
\bid{doi={10.1214/105051607000000104}, issn={1050-5164}, mr={2326236}}
\bptok{imsref}%
\end{barticle}
%
\endbibitem

\bibitem{HVV}
%
\begin{bincollection}[mr]
\bauthor{\bsnm{Hayes},~\bfnm{Thomas~P.}\binits{T.~P.}},
\bauthor{\bsnm{Vera},~\bfnm{Juan~C.}\binits{J.~C.}} \AND
\bauthor{\bsnm{Vigoda},~\bfnm{Eric}\binits{E.}}
(\byear{2007}).
\btitle{Randomly coloring planar graphs with fewer colors than the maximum
degree}.
In \bbooktitle{S{TOC}'07---{P}roceedings of the 39th {A}nnual {ACM} {S}ymposium
on {T}heory of {C}omputing}
\bpages{450--458}.
\bpublisher{ACM}, \baddress{New York}.
\bid{doi={10.1145/1250790.1250857}, mr={2402470}}
\bptok{imsref}%
\end{bincollection}
%
\endbibitem

\bibitem{HV}
%
\begin{barticle}[mr]
\bauthor{\bsnm{Hayes},~\bfnm{Thomas~P.}\binits{T.~P.}} \AND
\bauthor{\bsnm{Vigoda},~\bfnm{Eric}\binits{E.}}
(\byear{2006}).
\btitle{Coupling with the stationary distribution and improved
sampling for
colorings and independent sets}.
\bjournal{Ann. Appl. Probab.}
\bvolume{16}
\bpages{1297--1318}.
\bid{doi={10.1214/105051606000000330}, issn={1050-5164}, mr={2260064}}
\bptok{imsref}%
\end{barticle}
%
\endbibitem

\bibitem{Jerrum}
%
\begin{barticle}[mr]
\bauthor{\bsnm{Jerrum},~\bfnm{Mark}\binits{M.}}
(\byear{1995}).
\btitle{A very simple algorithm for estimating the number of
{$k$}-colorings of
a low-degree graph}.
\bjournal{Random Structures Algorithms}
\bvolume{7}
\bpages{157--165}.
\bid{doi={10.1002/rsa.3240070205}, issn={1042-9832}, mr={1369061}}
\bptok{imsref}%
\end{barticle}
%
\endbibitem

\bibitem{Jonasson}
%
\begin{barticle}[mr]
\bauthor{\bsnm{Jonasson},~\bfnm{Johan}\binits{J.}}
(\byear{2002}).
\btitle{Uniqueness of uniform random colorings of regular trees}.
\bjournal{Statist. Probab. Lett.}
\bvolume{57}
\bpages{243--248}.
\bid{doi={10.1016/S0167-7152(02)00054-8}, issn={0167-7152}, mr={1912082}}
\bptok{imsref}%
\end{barticle}
%
\endbibitem

\bibitem{LawlerSokal}
%
\begin{barticle}[mr]
\bauthor{\bsnm{Lawler},~\bfnm{Gregory~F.}\binits{G.~F.}} \AND
\bauthor{\bsnm{Sokal},~\bfnm{Alan~D.}\binits{A.~D.}}
(\byear{1988}).
\btitle{Bounds on the {$L\sp2$} spectrum for {M}arkov chains and {M}arkov
processes: A generalization of {C}heeger's inequality}.
\bjournal{Trans. Amer. Math. Soc.}
\bvolume{309}
\bpages{557--580}.
\bid{doi={10.2307/2000925}, issn={0002-9947}, mr={0930082}}
\bptok{imsref}%
\end{barticle}
%
\endbibitem

\bibitem{Ledoux}
%
\begin{bbook}[mr]
\bauthor{\bsnm{Ledoux},~\bfnm{Michel}\binits{M.}}
(\byear{2001}).
\btitle{The Concentration of Measure Phenomenon}.
\bseries{Mathematical Surveys and Monographs}
\bvolume{89}.
\bpublisher{Amer. Math. Soc.}, \baddress{Providence, RI}.
\bid{mr={1849347}}
\bptok{imsref}%
\end{bbook}
%
\endbibitem

\bibitem{LPW}
%
\begin{bbook}[mr]
\bauthor{\bsnm{Levin},~\bfnm{David~A.}\binits{D.~A.}},
\bauthor{\bsnm{Peres},~\bfnm{Yuval}\binits{Y.}} \AND
\bauthor{\bsnm{Wilmer},~\bfnm{Elizabeth~L.}\binits{E.~L.}}
(\byear{2009}).
\btitle{Markov Chains and Mixing Times}.
\bpublisher{Amer. Math. Soc.}, \baddress{Providence, RI}.
\bid{mr={2466937}}
\bptok{imsref}%
\end{bbook}
%
\endbibitem

\bibitem{Molloy}
%
\begin{barticle}[mr]
\bauthor{\bsnm{Lucier},~\bfnm{B.}\binits{B.}} \AND
\bauthor{\bsnm{Molloy},~\bfnm{M.}\binits{M.}}
(\byear{2011}).
\btitle{The {G}lauber dynamics for colorings of bounded degree trees}.
\bjournal{SIAM J. Discrete Math.}
\bvolume{25}
\bpages{827--853}.
\bid{doi={10.1137/090779516}, issn={0895-4801}, mr={2817533}}
\bptok{imsref}%
\end{barticle}
%
\endbibitem

\bibitem{Molloy2}
%
\begin{bincollection}[mr]
\bauthor{\bsnm{Lucier},~\bfnm{Brendan}\binits{B.}},
\bauthor{\bsnm{Molloy},~\bfnm{Michael}\binits{M.}} \AND
\bauthor{\bsnm{Peres},~\bfnm{Yuval}\binits{Y.}}
(\byear{2009}).
\btitle{The {G}lauber dynamics for colourings of bounded degree trees}.
In \bbooktitle{Approximation, Randomization, and Combinatorial Optimization. Algorithms and Techniques}.
\bseries{Lecture Notes in Computer Science}
\bvolume{5687}
\bpages{631--645}.
\bpublisher{Springer}, \baddress{Berlin}.
\bid{mr={2551034}}
\bptok{imsref}%
\end{bincollection}
%
\endbibitem

\bibitem{Martinelli-lecturenotes}
%
\begin{bincollection}[mr]
\bauthor{\bsnm{Martinelli},~\bfnm{Fabio}\binits{F.}}
(\byear{1997}).
\btitle{Lectures on {G}lauber dynamics for discrete spin models}.
In \bbooktitle{Lectures on Probability Theory and Statistics ({S}aint-{F}lour,
1997)}.
\bseries{Lecture Notes in Math.}
\bvolume{1717}
\bpages{93--191}.
\bpublisher{Springer}, \baddress{Berlin}.
\bid{mr={1746301}}
\bptnote{check year}%
\bptok{imsref}%
\end{bincollection}
%
\endbibitem

\bibitem{MSW}
%
\begin{barticle}[mr]
\bauthor{\bsnm{Martinelli},~\bfnm{Fabio}\binits{F.}},
\bauthor{\bsnm{Sinclair},~\bfnm{Alistair}\binits{A.}} \AND
\bauthor{\bsnm{Weitz},~\bfnm{Dror}\binits{D.}}
(\byear{2004}).
\btitle{Glauber dynamics on trees: Boundary conditions and mixing time}.
\bjournal{Comm. Math. Phys.}
\bvolume{250}
\bpages{301--334}.
\bid{doi={10.1007/s00220-004-1147-y}, issn={0010-3616}, mr={2094519}}
\bptok{imsref}%
\end{barticle}
%
\endbibitem

\bibitem{MSW:soda}
%
\begin{barticle}[mr]
\bauthor{\bsnm{Martinelli},~\bfnm{Fabio}\binits{F.}},
\bauthor{\bsnm{Sinclair},~\bfnm{Alistair}\binits{A.}} \AND
\bauthor{\bsnm{Weitz},~\bfnm{Dror}\binits{D.}}
(\byear{2007}).
\btitle{Fast mixing for independent sets, colorings, and other models on
trees}.
\bjournal{Random Structures Algorithms}
\bvolume{31}
\bpages{134--172}.
\bid{doi={10.1002/rsa.20132}, issn={1042-9832}, mr={2343716}}
\bptok{imsref}%
\end{barticle}
%
\endbibitem

\bibitem{Upfal}
%
\begin{bbook}[mr]
\bauthor{\bsnm{Mitzenmacher},~\bfnm{Michael}\binits{M.}} \AND
\bauthor{\bsnm{Upfal},~\bfnm{Eli}\binits{E.}}
(\byear{2005}).
\btitle{Probability and Computing: Randomized Algorithms and
Probabilistic Analysis}.
\bpublisher{Cambridge Univ. Press}, \baddress{Cambridge}.
\bid{mr={2144605}}
\bptok{imsref}%
\end{bbook}
%
\endbibitem

\bibitem{MosselSly}
%
\begin{barticle}[mr]
\bauthor{\bsnm{Mossel},~\bfnm{Elchanan}\binits{E.}} \AND
\bauthor{\bsnm{Sly},~\bfnm{Allan}\binits{A.}}
(\byear{2010}).
\btitle{Gibbs rapidly samples colorings of {$G(n,d/n)$}}.
\bjournal{Probab. Theory Related Fields}
\bvolume{148}
\bpages{37--69}.
\bid{doi={10.1007/s00440-009-0222-x}, issn={0178-8051}, mr={2653221}}
\bptok{imsref}%
\end{barticle}
%
\endbibitem

\bibitem{RSVVY}
%
\begin{binproceedings}[mr]
\bauthor{\bsnm{Restrepo},~\bfnm{Ricardo}\binits{R.}},
\bauthor{\bsnm{Stefankovic},~\bfnm{Daniel}\binits{D.}},
\bauthor{\bsnm{Vera},~\bfnm{Juan~C.}\binits{J.~C.}},
\bauthor{\bsnm{Vigoda},~\bfnm{Eric}\binits{E.}} \AND
\bauthor{\bsnm{Yang},~\bfnm{Linji}\binits{L.}}
(\byear{2011}).
\btitle{Phase transition for {G}lauber dynamics for independent sets
on regular
trees}.
In \bbooktitle{Proceedings of the {T}wenty-{S}econd {A}nnual {ACM}-{SIAM}
{S}ymposium on {D}iscrete {A}lgorithms}
\bpages{945--956}.
\bpublisher{SIAM}, \baddress{Philadelphia, PA}.
\bid{mr={2857176}}
\bptok{imsref}%
\end{binproceedings}
%
\endbibitem

\bibitem{SinclairJerrum}
%
\begin{barticle}[mr]
\bauthor{\bsnm{Sinclair},~\bfnm{Alistair}\binits{A.}} \AND
\bauthor{\bsnm{Jerrum},~\bfnm{Mark}\binits{M.}}
(\byear{1989}).
\btitle{Approximate counting, uniform generation and rapidly mixing {M}arkov
chains}.
\bjournal{Inform. Comput.}
\bvolume{82}
\bpages{93--133}.
\bid{doi={10.1016/0890-5401(89)90067-9}, issn={0890-5401}, mr={1003059}}
\bptok{imsref}%
\end{barticle}
%
\endbibitem

\bibitem{Sly}
%
\begin{barticle}[mr]
\bauthor{\bsnm{Sly},~\bfnm{Allan}\binits{A.}}
(\byear{2009}).
\btitle{Reconstruction of random colourings}.
\bjournal{Comm. Math. Phys.}
\bvolume{288}
\bpages{943--961}.
\bid{doi={10.1007/s00220-009-0783-7}, issn={0010-3616}, mr={2504861}}
\bptok{imsref}%
\end{barticle}
%
\endbibitem

\bibitem{Vigoda}
%
\begin{barticle}[mr]
\bauthor{\bsnm{Vigoda},~\bfnm{Eric}\binits{E.}}
(\byear{2000}).
\btitle{Improved bounds for sampling colorings}.
\bjournal{J. Math. Phys.}
\bvolume{41}
\bpages{1555--1569}.
\bid{doi={10.1063/1.533196}, issn={0022-2488}, mr={1757969}}
\bptok{imsref}%
\end{barticle}
%
\endbibitem

\end{thebibliography}
\end{document}